\documentclass[pdflatex,sn-mathphys-num]{sn-jnl}

\usepackage{graphicx}%
\usepackage{multirow}%
\usepackage{amsmath,amssymb,amsfonts}%
\usepackage{amsthm}%
\usepackage{mathrsfs}%
\usepackage[title]{appendix}%
\usepackage{xcolor}%
\usepackage{textcomp}%
\usepackage{manyfoot}%
\usepackage{booktabs}%
\usepackage{algorithm}%
\usepackage{algorithmicx}%
\usepackage{algpseudocode}%
\usepackage{listings}%
\usepackage{mathrsfs}%
\usepackage{fix-cm}
\newcommand{\mathbbm}[1]{\text{\usefont{U}{bbm}{m}{n}#1}}

\theoremstyle{thmstyleone}%
\newtheorem{theo}{Theorem}
\newtheorem{asum}{Assumption}%
\newtheorem{defi}{Definition}%
\newtheorem{lem}{Lemma}%

\theoremstyle{thmstyletwo}%

\theoremstyle{thmstylethree}%

\begin{document}

\title[Article Title]{Strong convergence of a semi tamed scheme for stochastic differential algebraic equation under  non-global Lipschitz coefficients.}


\author[1]{\fnm{Guy} \sur{Tsafack}}\email{Guy.Tsafack@hvl.no (guytsafack4@gmail.com)}

\author[1,2]{\fnm{Antoine } \sur{Tambue}}\email{antoine.tambue@hvl.no (antonio@aims.ac.za)}


\affil[1]{\orgdiv{Department of Computing Mathematics and Physics}, \orgname{Western Norway University of Applied Sciences}, \orgaddress{\street{Inndalsveien 28},  \postcode{5063 Bergen}, \state{} \country{Norway}}}

\affil[2]{\orgdiv{The Department of Mathematics \& Applied Mathematic}, \orgname{University of Cape
Town}, \orgaddress{\street{Private Bag }, \postcode{7701 Rondebosch} , \city{Cape Town},\country{ South Africa}}}




\abstract{ We are investigating the first  strong convergence  analysis  of a numerical method for  stochastic differential algebraic equations (SDAEs) under a non-global Lipschitz setting. It is well known that the explicit Euler scheme fails to converge strongly to the exact solution of a stochastic differential equation (SDEs) when at least one of the coefficients grows superlinearly. The problem becomes more challenging in the case of stochastic differential-algebraic equations (SDAEs) due to the singularity of the matrix. To address this, 
    we build a new scheme  called  the semi-implicit tamed method for  SDAEs and provide its strong convergence result under non-global Lipschitz setting.  In other words, the linear component of the drift term  is approximated implicitly, whereas its nonlinear component is tamed and approximated explicitly. We show that this method  strongly converges with order $\frac{1}{2}$ to the exact solution. To prove this strong convergence  result, we first derive an equivalent scheme,  that we call the  dual tamed scheme, which is more suitable for mathematical analysis and is associated with the inherent stochastic differential equation obtained by eliminating the constraints from the original SDAEs. To demonstrate the effectiveness of the proposed scheme, numerical simulations are performed, confirming that the theoretical findings are consistent with the numerical results.}




\maketitle
\section{Introduction}
	We are interested in the numerical method for solving the following stochastic differential algebraic
	equation: 
	\begin{equation}\label{equa1}
		A(t)dX(t)=\left[B(t)X(t)+f(t,X(t))\right]dt + g(t,X(t))dW(t), \quad t\in \left[0,T\right] ,\quad X(0)=X_0.
	\end{equation}
	Here $T>0,\quad T\ne \infty$, $ A\in \mathbb{R}^{d\times d}$ is a singular matrix, $B\in  \mathbb{R}^{d\times d} $ is a matrix.
	The functions  $f:\left[0,T\right]\times\mathbb{R}^d  \to \mathbb{R}^d  $ (drift) and $g:\left[0,T\right]\times\mathbb{R}^d \to \mathbb{R}^{d\times m} $ (diffusion) are Borel functions on $\left[0,T\right]\times\mathbb{R}^d$.
    The process $W(\cdot)$ is an m-dimensional  Wiener process defined on the probability space $(\Omega, \mathcal{F},\mathbb{P})$ with the natural filtration  $\left( \mathcal{F}_t\right)_{t\geq 0} $.\\
	The unknown function  $X(\cdot)$ is a vector-valued stochastic process of dimension  $d$ and depends both on time $ t\in \left[0,T \right] $ and a sample $\omega \in \Omega$. To ease the notation, the argument $\omega$ is omitted.\\

    Differential-algebraic equations (DAEs) are frequently used to model the dynamic behavior of mechanical multi-body systems, electric circuits, and systems in process engineering, see for example \cite{daoutidis2015daes} and references therein.
    Today, with the influence of some external forces in the systems, those models have been updated by using the stochastic differential-algebraic equations (SDAEs) \cite{kupper2012runge,Renate,tsafack2025pathwise}. Indeed, the SDAEs are the general form of the stochastic differential equations (SDEs). As for SDEs, there is no hope to find an analytical or exact solution for realistic SDAEs, therefore numerical algorithms are currently the only alternative techniques that can provide approximated solutions.
    It is well known that the explicit Euler scheme and semi-implicit Euler-type schemes do not converge strongly to the exact solution of SDEs when at least one of the coefficients $f$ and $g$ grows superlinearly (see, for example,\cite{hutzenthaler2013divergence, hutzenthaler2011strong}). 
    Implicit Euler-type schemes for SDEs are very stable and converge strongly when at least  one the functions $f$ and $g$ grows superlinearly. Such implicit schemes require the resolution of  systems of nonlinear algebraic equations at every time step,  which is very costly, mostly in high dimension. Recently explicit schemes based on tamed techniques have been developed and analysed strongly under non-global Lipschitz setting (see for example \cite{hutzenthaler2012strong, tambue2019strong} and references therein).  In contrast to SDEs where numerous algorithms have been developed and analyzed strongly under non-global Lipschitz coefficients,  numerical algorithm that converges strongly for SDAEs under non-global Lipschitz setting has been lacked in the literature to the best of our knowledge.
   
 Recently, the pathwise convergence  for SDAEs has been established using a semi-implicit approach under non-global Lipschitz conditions \cite{tsafack2025pathwise}. Unfortunately, pathwise convergence is not enough in many applications \cite{hutzenthaler2012strong}, and many numerical schemes that converge pathwise for SDEs  may not converge strongly. 
The goal of this paper is to construct and analyze the first numerical algorithm for SDAEs of type \eqref{equa1} that converges strongly  under non-global Lipschitz coefficients. 
 Our novel scheme is based on semi-tamed technique \cite{hutzenthaler2012strong, tambue2019strong}. Indeed the linear component of the drift term  is approximated implicitly, whereas its nonlinear component is tamed and approximated explicitly. We achieved the standard strong order rate $\frac{1}{2}$.
Note that singular matrices in
these equations pose considerable challenges for theoretical analysis and simulations. We have developed a tamed method based on the semi-implicit Euler scheme to overcome those difficulties. In fact, the tamed technique is applied only to the nonlinear part of the drift of our SDAEs, while the linear part is implicitly approximated with the main duty of covering the singularity of the matrix $A(t)$.     

This paper is organized as follows. In the second section, we present key definitions and important assumptions that will be useful in our analysis. The third section focuses on our novel numerical scheme based on the semi-implicit tamed technique. In the fourth section, we provide the strong convergence of our novel numerical method. Finally, we end the paper by presenting some numerical results in order to verify our theoretical results.

\section{Notations, assumptions, and definitions}
 Throughout this work,  $(\Omega, \mathcal{F},\mathbb{P})$ denotes a complete probability space with the natural 
filtration $\left( \mathcal{F}_t\right)_{t\geq 0} $. We define,  $\left\|x \right\|^2=\sum_{i=1}^{d}\left| x_i\right|^2  $ , for any vector $x\in \mathbb{R}^d$ and   $\left| B \right|_{1}=\max_{j}\sum_{i}\left|b_{i,j}\right| $ is the matrix norm, for any matrix $B=(b_{i,j})_{i,j=1}^{d, m}$. 

		\begin{defi}\label{def1}
			A strong solution of \eqref{equa1}  is a process $X(\cdot) = (X(t))_{t\in \left[0,T \right] }$ with continuous sample paths that respects the following conditions
			\begin{itemize}
				\item [(i)]  $X(\cdot)$ is adapted to the filtration $\left\lbrace \mathcal{F}_t\right\rbrace_{t\in \left[0,T\right] } $,
				\item [(ii)] $\int_{0}^t|f_i(s,X(s))|ds< \infty$ a.s, for all   $ i=1,..., d$,  $t\in \left[0,T\right]$,
				\item[ (iii)] $\int_{0}^tg_{ij}^2(s,X(s))dW_j(s)< \infty$ a.s, for all  $i=1,...,d; j=1,...,m$, $t\in \left[0,T\right]$,
				\item[(iv)]  $X(\cdot) = (X(t))_{t\in \left[0,T \right] }$ satisfies the following equation
				\begin{align}
					\label{reps}
					A(t)X(t)&=A(0)X_0+\int_0^tA'(s)X(s)+B(s)X(s)+f(s, X(s))ds\nonumber\\
					&+\int_0^tg(s, X(s))dW(s).
				\end{align} 
			\end{itemize}		
		\end{defi} 
		Note that for having the representation \eqref{reps}, we have applied the It\^o lemma to the  function $k(t,x)=A(t)x$. This means that $dk(t,x)=A'(t)xdt+A(t)dx$.\\
        
        Note also that any SDAEs can be split to an algebraic equations (AEs) and a stochastic differential equations (SDEs).
        \begin{defi}\label{def2}
        The SDAEs \eqref{equa1} is said to be of index-1 if we have the following conditions:
        \begin{itemize}
				\item [(i)] $Img(t,X)\subseteq ImA(t)$, for all $X\in \mathbb{R}^d$ and $t\in [0,T]$. This ensures that  the noise sources do not appear into the
constraints AEs.
                \item [(ii)]  the constraints AEs are globally uniquely solvable.
                			\end{itemize}	
        \end{defi} 
        The well posedness of the SDAE \eqref{equa1} is investigated under the following hypothesis.
        \newpage
 \begin{asum}\label{asum1} We assume that
	\begin{enumerate} 
		\item[(A.1)] The function $f$  satisfies the local  Lipschitz condition i.e. for all $q\geq 1$ there exists $L_q>0$ such that
		\begin{equation*}
			\left\| f(X)-f(Y)\right\| \leq L_q \left\|Y- X\right\|,~\forall  X, Y\in \mathbb{R}^d ~with ~ \left\|X \right\|\lor \left\|Y \right\|<q ;
		\end{equation*}
and $g$ satisfies the global Lipschitz condition
			\begin{equation*}
			\left| g(X)-g(Y)\right|_F \leq L \left\|Y- X\right\|, 	~\forall   X, Y\in \mathbb{R}^d.~~~~~~~~~~~~~~~~~~~~~~~~~~~
		\end{equation*}

\item[(A.2)]	The function $f$ satisfies the one side Lipschitz condition: For $X,Y\in \mathbb{R}^d$ there exists $C\geq 0$ such that
	\begin{equation*}
	\left\langle X-Y, f(X)-f(Y)\right\rangle \leq C\left\| X-Y\right\| ^2~~~~~~~~~~~~~~~~~~~~~~~~~~~~~~~~~~
		\end{equation*}
	
	\item[(A.3)] 	The function $f$ satisfies the following super-linear growth condition: there exist $C\geq 0 $ and $c\geq 0 $ such that
		\begin{equation*}
		\left\| f(X)-f(Y)\right\|  \leq C\left(1+\left\| X\right\|^c+\left\| Y\right\|^c \right) \left\| X-Y\right\|, ~X,Y\in \mathbb{R}^d~~~~
	\end{equation*}
	\item[(A.4)] There exists $M>0$   such that $\mathbb{E}\left\|X_0 \right\|^2\leq M $

		\end{enumerate}
  \end{asum}
\begin{theo}\label{theo1}
	We assume that the SDAE \eqref{equa1} is an index 1 SDAEs and the conditions (A.1) and  (A.2)  of Assumption \ref{asum1} are satisfied,   then  the SDAE \eqref{equa1} has	a unique solution that satisfies the relation
	 \begin{equation}
	 	\mathbb{E}(\|X(t)\|^2)\leq C, ~t\in [0,T],~~~~~~~~~~~~~~~~~~~~~~~~~~~~~~~~~~~~~~~~~~~~~~~
	 \end{equation}
    where $C$ is a  positive constant.
\end{theo}
\begin{proof}  From  \cite[Theorem 1]{serea2025existence} it is enough to prove  that the coefficient of  \eqref{equa1} satisfies the following monotone condition
  \begin{align}\label{mon}
  	\langle (A^-(t)A(t)X)^T, A^-(t)(B(t)X(t)+f(X))\rangle&+ \frac{1}{2}|A^-(t)g(X)|_1^2\nonumber \\&\leq C(1+\|X\|^2),~X\in \mathbb{R}^d,
  \end{align}
  where $A^-$ is the pseudo-inverse matrix of the matrix $A$.  From Assumption \ref{asum1} (A.1), we have 
  \begin{align*}
  		\left| g(X)-g(0)\right|_1 &\leq L \left\| X\right\|, 	~\forall   X\in \mathbb{R}^d.~~~~~~~~~~~~~~~~~~~~~~~~~~~~~~~~~~~~~~~~~
  \end{align*} 
  This means that:
   \begin{align*}
  	\left| g(X)\right|_1&=\left| g(X)-g(0)+g(0)\right|_1\leq \left| g(X)-g(0)\right|_1+\left|g(0) \right|_1.~~~~~~~~~~~~~~~~~~~~\nonumber\\
   &\leq\left|g(0) \right|_1 + L \left\| X\right\|, 	~\forall   X\in \mathbb{R}^d
  \end{align*} 
and then 
 \begin{align}\label{g}
	\left| g(X)\right|_1^2 &\leq C_1(1+ \left\| X\right\|^2), 	~\forall   X\in \mathbb{R}^d.~~~~~~~~~~~~~~~~~~~~~~~~~~~~~~~~~~~~~~~~~
\end{align}
From Assumption \ref{asum1} (A.2), we have 
	$\left\langle X, f(X)-f(0)\right\rangle \leq C\left\| X\right\| ^2$.\\
By using the relation $2ab\leq a^2+b^2$, we obviously have
	\begin{align}\label{f}
	\left\langle X, f(X)\right\rangle &\leq C\left\| X\right\| ^2+\left\langle X, f(0)\right\rangle\leq C_1(1+\left\| X\right\| ^2).
\end{align}
Finally, we have
  \begin{align}
  		\langle (A^-(t)A(t)X)^T, &A^-(t)(B(t)X+f(X))\rangle+ \frac{1}{2}|A^-g(X)|_1^2\nonumber\\&\leq 
  		|	\langle (A^-AX)^T, A^-f(X)\rangle|\nonumber\\&+	|\langle (A^-AX)^T, A^-(B(t)X(t)\rangle|+ \frac{1}{2}|A^-g(X)|_1^2\nonumber\\
  		&\leq | A^-|_1^2|A|_1	|B|_1(1+\|X\|^2	)\nonumber\\
  	&+| A^-|_1^2|A|_1	|	\langle X^T, f(X)\rangle|+ \frac{| A^-|_1^2}{2}|g(X)|_1^2.~~~~~~
  \end{align}
 From \eqref{f} and \eqref{g}, we obtain \eqref{mon}. We conclude this proof using \cite[Theorem 1]{serea2025existence}.
	\end{proof}
\section{Numerical scheme}
We consider the SDAE \eqref{equa1}. We apply the tamed technique
as in \cite{hutzenthaler2012strong} and \cite{tambue2019strong})  to the drift term  \footnote{ only on $f$} and  the implicit Euler method in the linear part of the drift and obtain the following scheme,  that we call the  semi-implicit tamed scheme
\begin{align}\label{equa2}
	A(t_n)X_{n+1}^N-A(t_n)X_n^N&=hB(t_n)X_{n+1}^N+\frac{f(X_n^N) h}{1+h\left\|f(X_n^N) \right\| }+g(X_n^N)\Delta W_n^N,\\
	X_0^N&=X_0, ~n=0,1,2,...,N-1,\nonumber 
\end{align}
where $h =\frac{T}{N}$ is the time step-size, $\Delta W_n^N=W(t_{n+1})-W(t_n)$, $t_{n+1}=t_n+h=(n+1)h,~n=0,1,2,3,..., N-1$, with   $t_0=0$ and  $N\in\mathbb{N}$ is the number of time subdivisions. 
Observe that equation \eqref{equa2} can be written as 
\begin{align}\label{equa3}
	(A(t_n)-h B(t_n))X_{n+1}^N&=A(t_n)X_{n}^N+\frac{f(X_n^N)h}{1+h\left\|f(X_n^N) \right\| }+g(X_n^N)\Delta W_n^N, \\
	X_0^N&=X_0, ~n=0,1,2,...,N-1.\nonumber
\end{align}

This paper investigates the strong convergence of the approximated solution  $X_n^N$, which solves \eqref{equa3}, to the exact solution $X$ of \eqref{equa1}.
Let $A^{-}$ be the pseudo-inverse of matrix $A$, and $P$, $R$, $Q$ be the projector matrices associated with $A$, such that $R(t)A(t)=0_{d\times d}, \text{ and } A^-(t)A(t)=P(t)=I-Q, ~t\in [0,T]$ (see \cite{serea2025existence}).

We make  the following assumption, which is important for our main result. 
\begin{asum}\label{asum2} We assume that
\begin{enumerate}
		\item[(A2.1)]  The matrices  $A,~ A^-,~P,~P',B $ are bounded and Lipschitz by the constant $K\geq 0$.
  
      \item[(A2.2)] The matrices $(A(t_n)+R(t_n)B(t_n))~ ,n\in\mathbb{N}$ and $S_{h}=A(t_n)-h B(t_n),~n\in\mathbb{N}$ are non-singular matrices, and their inverses are also bounded with the same constant $K>0$.
      
    
      \item[(A2.3)] The matrix  $I_{d} +hM_1(t_n),~n\in\mathbb{N}$ is non singular matrix and satisfies the relation $$|(I_{d}+hM_1(t_n))^{-1}|_1\leq \exp{(Kh)},\,\,n \in \{ 0,1,2,..., N-1 \},$$  where $K$ is a positive constant,  $I_{d}$ is the  identity matrix in $\mathbb{R}^d$ and   \begin{align*}
 M_1(t_n)&=-P'(t_n)+P'(t_n)[A(t_n)+R(t_n)B(t_n)]^{-1}R(t_n)B(t_n)\\
 &-A^-(t_n) B(t_n)+A^-(t_n) B(t_n)[A(t_n)+R(t_n)B(t_n)]^{-1}R(t_n)B(t_n).
			\end{align*}  Note that similar assumption can be found in \cite[Inequality (24)]{schurz2005convergence}. 
     
\end{enumerate}

\end{asum}
 The main result for this section is given in the following Theorem
\begin{theo}\label{theo2}
	Let $X_n^N: \Omega \longrightarrow \mathbb{R}^d,~ n\in \{0,1,..,N\}\text{ and } N\in \mathbb{N}$ be  the solution of equation \eqref{equa3}. Suppose that the Assumptions \ref{asum1} and \ref{asum2} hold and let $p \geq 1$, then we have
	\begin{align}
		&\sup_{N\in \mathbb{N}}\sup_{n\in \{ 0,1,2,...,N\}} \mathbb{E}[\|X_n^N\|^p]<+\infty.\label{equa28}~~~~~~~~~~~~~~~~~~~~~~~~
	\end{align}	
\end{theo}

The proof of this theorem needs some preliminary results.
\subsection{ Preliminary results}
The first lemma is the equivalence  numerical method  to the numerical method \eqref{equa2}
\begin{lem}\label{lem7} Suppose that the matrix $(A(t_n)+R(t_n)B(t_n)),n\in\mathbb{N}$ is a non singular matrix. Then, the scheme defined in equation \eqref{equa2} is equivalent to the following scheme
		\begin{equation}\label{equa31}
			\left \{
			\begin{array}{c c c}
				u_{n+1}-u_n &= P'(t_n)\left[ u_{n+1}+\hat{v}^N(t_{n+1},u_{n+1})\right] h+A^-(t_n) B(t_n)\left[u_{n+1}+\hat{v}^N(t_{n+1},u_{n+1})\right]h~~~~~~~~~~~~~\\
				\\
				&+\cfrac{A^-(t_n)f(t_n,u_{n}+\hat{v}^N(t_{n},u_{n}))h}{1+h\|f(t_n,u_{n}+\hat{v}^N(t_{n},u_{n}))\|} +A^-(t_n)g(t_n,u_{n}+\hat{v}^N(t_{n},u_{n}))\Delta W_n,\\
				\\
				\hat{v}^N(t_{n+1},u_{n+1})&=-(A(t_n)+R(t_n)B(t_n))^{-1}\left(R(t_n)B(t_n)u_{n+1}+\dfrac{R(t_n)f(t_n,u_{n}+\hat{v}^N(t_{n},u_{n}))}{1+h\|f(t_n,u_{n}+\hat{v}^N(t_{n},u_{n}))\|} \right),\\
				\\
        u_0&=P(0)\zeta,~~~~~~~~~~~~~~~~~~~~~~~~~~~~~~~~~~~~~~~~~~~~~~~~~~~~~~~~~~~~~~~~~~~~~~~~~~~~~~~~~~~~\\
        \hat{v}^N(0,u_0)&=Q(0)\zeta,~~~~~~~~~~~~~~~~~~~~~~~~~~~~~~~~~~~~~~~~~~~~~~~~~~~~~~~~~~~~~~~~~~~~~~~~~~~~~~~~~~~~\\
				\\
				X_{n+1}&=u_{n+1} +\hat{v}^N(t_{n+1},u_{n+1}),	~ n=0,1,...,N-1,~N\in\mathbb{N},~~~~~~~~~~~~~~~~~~~~~~~~~~~~~~~~
			\end{array}\right. 
		\end{equation}
		where the  matrix $A^{-}$ is the pseudo-inverse matrix of the matrix $A$ and,   $P$  and $R$, and $Q$ are projectors matrix associated to $A$, such that $R(t)A(t)=0_{d\times d}, \text{ and } A^-(t)A(t)=P(t)=I-Q, ~t\in [0,T]$ and the function $\hat{v}^N$ is global Lipschitz with respect the variable $u$.
	\end{lem}
	
	\begin{proof}
		The proof of this lemma uses the same logic or strategy as the proof of \cite[Lemma 12]{tsafack2025pathwise}
	\end{proof}

 The proof of the following lemma can be found in \cite[Lemma 3.8]{hutzenthaler2012strong} or in \cite[Lemma 3.13]{tambue2019strong}.
\begin{lem}\label{lem8}
	Let $k\in \mathbb{N}$ and let $Z_l^N:\Omega \longrightarrow 	  \mathbb{R}^{k\times m}$, $l\in \{0,1,...,N-1\},~N\in \mathbb{N}$ be a family of mappings such that $Z_l^N$ is $\mathcal{F}_{lT/N}/\mathcal{B}(\mathbb{R}^{k\times m})$-measurable for all $l\in \{0,1,...,N-1\},~N\in \mathbb{N}$. Then we have
	\begin{equation}
		\left\|\sup_{j\in \{ 0,~1,~2,...,~n\}}\sum_{l=0}^{j-1}\left\|Z_l^N\Delta W_l^N\right\|\right\|_{L^p(\Omega, \mathbb{R}^d)}\leq p\left(\sum_{l=0}^{n-1}\|Z_l^N\|^2_{L^p(\Omega, \mathbb{R}^{k\times m})}\frac{T}{N}\right)^\frac{1}{2},
	\end{equation}
	for all $p>1$ and where $\|Z_l^N\|^2_{L^p(\Omega, \mathbb{R}^{k\times m})}=\left(\mathbb{E}|Z_l^N|^{p}_F\right)^{\frac{2}{p}}.$
\end{lem}

 
 Our following lemma is the well-known Gronwall lemma.
	\begin{lem}\label{lem1t1}
		[The Discrete Gronwall Inequality]. Let $M$ be a positive integer. Let $u_k$ and $v_k$ be nonnegative numbers for $k = 0, 1,
		. . . , M.$
		$$\text{If  } ~~~~u_k\leq u_0+\sum_{j=0}^{k-1}v_ju_j;~k = 0, 1,
		. . . , M.~~~~~~~~~~~~~~~~~~~~~~~~~~~~~~~~~~~~$$
		Then $$u_k\leq u_0\exp{\left(\sum_{j=0}^{k-1}v_j\right)};~k = 0, 1,
		. . . , M. ~~~~~~~~~~~~~~~~~~~~~~~~~~~~~~~~$$
		
	\end{lem}
	The proof can be found in \cite[Lemma 3.4 ]{mao2013strong} or in \cite{mao2006stochastic}.

	
	\begin{lem}\label{lem1} For all $a,~b \geq 1$ , the following inequality holds
		$$a+b\leq e^{ab}.$$
	\end{lem}	
    
 The following lemma is  also important in our analysis
\begin{lem}\label{lem2}
	Let $X_n^N:\Omega\to \mathbb{R}^d$, $D_n^N: \Omega \to \left[ 0,T\right] $ and $\Omega_n^N\in \Omega$ for $n\in \left\lbrace 0,1,2, ..., N \right\rbrace $ and $N\in \mathbb{N}$ given;  then we have 
	\begin{equation}\label{equa4}
		\mathbbm{1}_{{\Omega}_n^N}\left\| X_n^N\right\| \leq D_n^N,\text{ for } n\in \left\lbrace 0,1,2,...,N \right\rbrace \text{ and all } N\in \mathbb{N} ,
	\end{equation}
where $$\Omega_n^N:=\left\lbrace \omega\in \Omega: \sup_{k\in \{ 0,~1,~2,...,~n-1\} }  D_k^N(\omega)\leq N^{\frac{1}{2c}};~ \sup_{k\in \{ 0,~1,~2,...,~n-1\}}  \left\| \Delta W_k^N(\omega)\right\|\leq 1\right\rbrace .   $$ for all $~n\in \{ 0,~1,~2,...,~N-1\}$, $N\in \mathbb{N}$ and the constant $c\geq 0$ is defined in Assumption \ref{asum1} (A.3).
\end{lem}
\begin{proof} 
	From the definition of $\Omega_n^N$ note that  $\left\| \Delta W_n^N(\omega)\right\|\leq 1$ on $\Omega_{n+1}^N$ for all  $n\in \{ 0,~1,~2,...,~N-1\}$,  $N\in \mathbb{N}$.
 
 We recall equation \eqref{equa3} here as follows with $S_h=\left(A(t_n)-hB(t_n)\right)^{-1}$
	\begin{equation*}
		X_{n+1}^N=S_{h}AX_{n}^N+\frac{S_{h}f(X_n^N)h}{1+h\left\|f(X_n^N) \right\| }+S_{h}g(X_n^N)\Delta W_n^N,~X_0^N=X_0, ~n=0,1,2,...,N-1.
	\end{equation*}
Taking the norm on both sides, we obtain
	\begin{equation*}
	\left\| X_{n+1}^N\right\| \leq \left\| S_{h}AX_{n}^N\right\| +\frac{\left\| S_{h}f(X_n^N)\right\| h}{1+h\left\|f(X_n^N) \right\| }+\left\|S_{h} g(X_n^N)\Delta W_n^N\right\|,~X_0^N=X_0, ~n=0,1,2,...,N-1.
\end{equation*}
	But $\left\|\Delta W_n^N\right|\leq 1$, $h=\frac{T}{N}$; $\frac{1}{1+h\left\| f((X_n^N)\right\| }\leq 1,~n=0,1,2,...,N-1, ~~N\in \mathbb{N}.$ 
 
 This means that by using Assumption \ref{asum1} and Assumption \ref{asum2}, we have
		\begin{align}\label{equa5}
		\left\| X_{n+1}^N\right\|& \leq \left| S_{h}\right|_1 \left| A\right| _1\left\| X_{n}^N\right\| +T\left| S_{h}\right| _1\left\| f(X_n^N)-f(0)+f(0)\right\|~~~~~~~~~~~~~~~~~~~~~~\nonumber\\
  &+\left|S_{h}\right| _1 \left| g(X_n^N)-g(0)+g(0)\right| _1\nonumber\\
		& \leq 1+ K\left| A\right| _1+2KTC+ KT\left\| f(0)\right\| +LK+K\left| g(0)\right| _1\nonumber\\
			& \leq \lambda, \text{ for all }~n\in \{0,1,2,...,N-1\},~ N\in \mathbb{N} ,
	\end{align}
on $\Omega_{n+1}^N\cap \{\omega\in \Omega: \left\|  X_{n}^N(\omega)\right\|\leq 1 \}$ for all $n\in \{0,1,2,...,N-1\}$ and $N\in \mathbb{N}$. \\
The constant $\lambda$ is such that  for all $n\in \{0,1,...,N\}, ~N\in \mathbb{N}$
$$\lambda  \geq ( 1+ K\left| A\right| _1+2KTC+ KT\left\| f(0)\right\| +LK+K\left| g(0)\right| _1)^2 \text{  and } 2\lambda \|\Delta W_n^N\|^2\geq 1,$$ where the constants $K$,$L$ and $C$ are defined in Assumption \ref{asum1} and \ref{asum2}.\\
Using Cauchy-Schwartz inequality and H\"older inequality  in equation \eqref{equa3}, we obtain
	\begin{align}\label{equa6}
	\left\| X_{n+1}^N\right\| ^2&=\left\| S_{h}AX_{n}^N+\frac{S_{h}f(X_n^N)h}{1+h\left\|f(X_n^N) \right\| }+S_{h}g(x_n^N)\Delta W_n^N\right\| ^2\nonumber\\
	& \leq \left| S_{h}\right| _1^2\left| A\right| _1^2\left\| X_{n}^N\right\|^2 +h^2\left| S_{h}\right| _1^2\left\| f(X_n^N)\right\| ^2+\left| S_{h}\right| _1^2\left| g(x_n^N)\right| _1^2\left\|\Delta  W_n^N\right\|^2\nonumber\\
	&+2h\left\langle S_{h}AX_n^N, \frac{S_{h}f(X^N_n)}{1+h\left\| f(X_n^N)\right\| }\right\rangle +2\left\langle S_{h}AX_n^N, S_{h}g(X^N_n)\Delta W_n^N\right\rangle\nonumber\\
	& +2\left\langle S_{h}g(X_n^N)\Delta W_n^N, \frac{S_{h}f(X^N_n)h}{1+h\left\| f(X_n^N)\right\| }\right\rangle\nonumber\\
	&\leq \left| S_{h}\right| _1^2\left| A\right| _1^2\left\| X_{n}^N\right\|^2 +2h^2\left| S_{h}\right| _1^2\left\| f(X_n^N)\right\| ^2+2\left| S_{h}\right| _1^2\left| g(x_n^N)\right| _1^2\left\|\Delta  W_n^N\right\|^2\nonumber\\
	&+2h\left\langle S_{h}AX_n^N, \frac{S_{h}f(X^N_n)}{1+h\left\| f(X_n^N)\right\| }\right\rangle +2\left\langle S_{h}AX_n^N, S_{h}g(X^N_n)\Delta W_n^N\right\rangle.
\end{align}
On $\Omega$ for all $n \in \{0, 1, . . . , N -1\}$ and all $N \in \mathbb{N}$. 

Additionally, using  Assumption \ref{asum1} and  \ref{asum2}  allows to have
\begin{align}\label{equa7}
	\left\langle S_{h}AX, \frac{S_{h}f(X)}{1+h\left\| f(X)\right\| }\right\rangle&\leq \left\langle S_{h}AX, S_{h}\left[ f(X)-f(0)+f(0)\right] \right\rangle\nonumber\\
	&\leq  (CK^2\left| A\right| _1+K^2\left| A\right| _1 \left\| f(0)\right\|)\left\|X\right\|^2,
\end{align}
for all $X\in \mathbb{R}^d$ with $\left\|X\right\|\geq 1$.\\
We define $\alpha_n^N, ~n\in \{0,1,2,...,N-1\}, ~N\in \mathbb{N}$ by
\begin{equation}\label{equa8}
	\alpha_n^N=\mathbbm{1}_{\{\left\|X_n^N \right\|\geq 1\} }\left\langle\frac{ S_{h}AX_n^N}{\left\|X^N_n \right\|}, \frac{S_{h}g(X^N_n)\Delta W_n^N}{\left\|X^N_n \right\| }\right\rangle,  ~n\in \{0,1,2,...,N-1\}, ~N\in \mathbb{N}.
\end{equation}
Using Assumption \ref{asum1}  we obtain
\begin{align}\label{equa9}
	\left|g(X) \right|_1^2&\leq (	\left|g(X)-g(0) \right|_1+
\left|g(0) \right|_1 )  ^2~~~~~~~~~~~~~~~~~~~~~~~~~~~~~~~~~~~~~~~~~~~~~~~~~~~~~~\nonumber\\
	&\leq (  L\left\|X \right\|+\left|g(0) \right|_1 ) ^2 \leq (  L+\left|g(0) \right|_1)^2\left\|X \right\|^2,
\end{align}
for all $X\in \mathbb{R}^d$ with $\left\|X\right\|\geq 1$.\\
Moreover, using (A1.3) in Assumption \ref{asum1} we have also
\begin{align}\label{equa10}
\left\|f(X) \right\| ^2&\leq(\left\|f(X) -f(0)\right\|+\left\| f(0)\right\| )^2~~~~~~~~~~~~~~~~~~~~~~~~~~~~~~~~~~~~~~~~~~~~~~~~~~\nonumber\\
& \leq\left[  C(1+\left\|X \right\|^c )\left\|X \right\|+\left\| f(0)\right\|\right] ^2\nonumber\\
& \leq\left[  2C +\left\| f(0)\right\|\right] ^2\left\|X \right\|^{2c+2}\leq N\left[  2C +\left\| f(0)\right\|\right] ^2\left\|X \right\|^{2}
\end{align}
for all $X \in \mathbb{R}^d$ with $1  \leq\left\|X \right\|\leq  N^{\frac{1}{2c}}$ and all $ N\in \mathbb{N}$.\\
Finally,  we can substitute equations \eqref{equa7},\eqref{equa8}, \eqref{equa9}, \eqref{equa10} in \eqref{equa6} and obtain 
\begin{align}\label{equa11}
\left\|X_{n+1}^N \right\|^2& \leq K^2\left|A\right| _1^2\left\|X_{n}^N \right\|^2+2K^2\frac{T^2}{N^2}\left[ N\left[  2C +\left\| f(0)\right\|\right] ^2\left\|X_{n}^N \right\|^{2} \right] \nonumber\\
&+2K^2 (  L+\left|g(0) \right|_1)^2\left\|X_{n}^N \right\|^2\left\|\Delta W_n^N\right\|^2\nonumber\\
&+2\frac{T}{N}\left[  (CK^2\left| A\right| _1+K\left| A\right| _1 \left\| f(0)\right\|)\left\|X_{n}^N\right\|^2\right] +2\alpha_n^N\left\|X_{n}^N \right\|^2\nonumber\\
&\leq \left\|X_{n}^N \right\|^2\left[ K^2\left| A\right| _1^2+\frac{2\lambda}{N}+2\lambda\left\|\Delta W_n^N\right\|^2+2\alpha_n^N \right].
\end{align}
As $K^2\left| A\right| _F^2\leq \lambda$ and $\lambda\geq 1$,  using the fact that $2\lambda \left\|\Delta W_n^N\right\|^2\geq 1$, we can apply Lemma \ref{lem1} in the inequality \eqref{equa11} and obtain
\begin{align}\label{equa11e}
\left\|X_{n+1}^N \right\|^2&\leq  \left\|X_{n}^N \right\|^2\exp{\left(  \lambda\left[\frac{2\lambda}{N}+2\lambda\left\|\Delta W_n^N\right\|^2+2\alpha_n^N \right]\right) },  ~n\in \{0,1,...,N-1\},~ N\in \mathbb{N}\nonumber.\\
& \text{Consequently }\nonumber\\
\left\|X_{n+1}^N \right\|&\leq  \left\|X_{n}^N \right\|\exp{\left( \lambda \left[ \frac{\lambda}{N}+\lambda\left\|\Delta W_n^N\right\|^2+\alpha_n^N\right] \right) },
\end{align}
on $\{\omega\in \Omega: 1\leq  \left\|X_{n}^N(\omega) \right\|\leq N^\frac{1}{2c},~for ~n\in \{0,1,...,N-1\}\}$ and $N\in \mathbb{N}$.\\

We define $D_n^N$ by
\begin{equation}\label{equa12}
D_n^N(\omega)=(\lambda+\left\| X_0\right\| )\exp{\left(\lambda^2+\lambda\sup_{u\in \{ 0,~1,~2,...,~n\}}\left( \sum_{k=u}^{n-1}\left( \lambda\left\|\Delta W_k^N (\omega)\right\|^2+\alpha_k^N(\omega)  \right) \right)\right)  },
\end{equation}
where $\omega\in \Omega$, $n\in \{0,1,...,N-1\}$ and $N\in \mathbb{N}$.
Let us define  the map $\tau_l^N$ as
\begin{align}\label{equa13}
\tau_l^N:&~~ \Omega\to \{-1,0,1,2, ..., l-1\}, l \in \{0,1, 2,..., N\}, N\in \mathbb{N} ~\nonumber\\
\omega&~~~ \longmapsto \max\left(\{-1\}\cap\{ \{0,1,...,l-1\}|\left\|X_n^N(\omega) \right\|\leq 1\}  \right) .~~~~~~~~~~~~~~~~~~~~~~
\end{align}
We are able to establish \eqref{equa4} by  combining equation \eqref{equa11e}, \eqref{equa12}  and \eqref{equa5}. This means that by induction
for $n\in  \{0, 1, . . . , N\}$ where $N \in \mathbb{N}$ is fixed , we can prove that

$$\mathbbm{1}_{{\Omega}_n^N}\left\| X_n^N\right\| \leq D_n^N,\text{ for } n\in \left\lbrace 0,1,2,...,N \right\rbrace \text{ and all } N\in \mathbb{N} .~~~~~~~~~~~~~~~~~~~~~~$$
For more details we refer the reader to   \cite[Lemma 3.1 , pp 1628]{hutzenthaler2012strong}.
\end{proof}

The next lemma shows that the variable $D_n^N$ is bounded.
\begin{lem}\label{lem3}
For all $p\in \left[1;+\infty\right) $	 we have the following result
\begin{equation}\label{equa19}
	\sup_{ \substack{N\in \mathbb{N},\\ N\geq 4p\lambda^2  }}\mathbb{E}\left[\exp{\left(p\lambda^2\sum_{k=0}^{N-1}\left\|\Delta W_k^N\right\|^2\right)}\right]<+\infty .~~~~~~~~~~~~~~~~~~~~~~~~~~~~~~~~~~~~
\end{equation}
\end{lem}
The proof follows almost the same steps as the proof in \cite[ Lemma 3.3]{hutzenthaler2012strong}.

\begin{lem}\label{lem4}
We consider the following function: $\alpha_n^N: \Omega\longrightarrow [0;+\infty) $ for $n\in \{0,1,...,N\},~~N\in \mathbb{N}$. Note that $\alpha_n^N$ is given in \eqref{equa8}.  Then we have the following result
\begin{equation}\label{equa }
	\sup_{z\in \{-1,1\}}\sup_{N\in \mathbb{N}}\mathbb{E}\left[\sup_{n\in \{ 0,1,2,...,N\}}\exp{\left(\lambda pz\sum_{k=0}^{n-1}\alpha_k^N\right)}\right]<+\infty,~p\in [1,+\infty).~~~~~~
\end{equation}	
\end{lem}
\begin{proof}
	From the proof of  \cite[Lemma 3.4 ]{hutzenthaler2012strong}, the stochastic process $z\sum_{k=0}^{n-1}\alpha_k^N, ~n\in \{ 0,1,..,N\}$ is an $(\mathcal{F}_{t_n})_{n\in \{0,1,2,...,N\}}$-martingale for $z\in \{-1,1\}$, therefore the stochastic process $z\sum_{k=0}^{n-1}\lambda \alpha_k^N, ~n\in \{ 0,1,..,N\}$ is also  an $(\mathcal{F}_{t_n})_{n\in \{0,1,2,...,N\}}$-martingale for $z\in \{-1,1\}$. Using  Doob's maximal inequality (see, \cite[Theorem 11.2]{klenke2013probability}), we obtain
	\begin{equation}
		\left\|\sup_{n\in \{0,1,...,N\}}\exp{\left(z\sum_{k=0}^{n-1}\lambda \alpha_k^N\right)}\right\|_{L^p(\Omega,\mathbb{R})}\leq \frac{p}{p-1}\left\|\exp{\left(z\sum_{k=0}^{N-1}\lambda\alpha_k^N\right)}\right\|_{L^p(\Omega,\mathbb{R})}.
	\end{equation}
	We also have  for $x\in \mathbb{R}^d$ with $\|x\|\ne 0$
	\begin{align*}
		\mathbb{E}&\left[\left|\lambda pz\mathbbm{1}_{\{\left\|x \right\|\geq 1\} }\left\langle\frac{ S_{h}Ax}{\left\|x \right\|}, \frac{S_{h}g(x)\Delta W_n^N}{\left\|x \right\| }\right\rangle\right|^2\right]\\
		=&\lambda^2 p^2\mathbb{E}\left[\mathbbm{1}_{\{\left\|x \right\|\geq 1\}}\frac{|(S_{h}g(x))^TS_{h} Ax\Delta W^N_k|^2}{\|x\|^4}\right]~~~~~~~~~~~~~~~~~~~~~~~~~~~~~~~~~~\\
		=&\frac{\lambda^2 p^2T}{N}\frac{\|(S_{h}g(x))^TS_{h} Ax\|^2}{\|x\|^4}\\
		\leq &\frac{|A|_F^2\lambda^2 p^2T|S_{h}|_F^4}{N}\frac{|g(x)|_1^2}{\|x\|^2}.
	\end{align*}
	 The rest of the proof follows the same steps as in \cite[Proof of Lemma 3.4]{hutzenthaler2012strong}.
\end{proof}

\begin{lem}\label{lem5}
[Uniformly bounded moments of the dominating stochastic processes]\\
Let $D_n^N: \Omega \longrightarrow [0,\infty)$ for $n\in \{0,1,...,N\},~N\in \mathbb{N}$ be defined by 
$$D_n^N(\omega)=(\lambda+\left\| X_0\right\| )\exp{\left(\lambda^2+\lambda\sup_{u\in \{ 0,~1,~2,...,~n\}}\left( \sum_{k=u}^{n-1}\left( \lambda\left\|\Delta W_k^N (\omega)\right\|^2+\alpha_k^N(\omega)  \right) \right)\right) } $$.
Then  for all $p\in [1,+\infty)$, we have
\begin{equation}
		\sup_{ \substack{N\in \mathbb{N},\\ N\geq 8\lambda^3 Tp }}\mathbb{E}\left(\sup_{n\in \{0,1,...,N\}}|D_n^N|^p\right)<\infty.
\end{equation}
\end{lem}
\begin{proof}
The proof of this lemma follows the same steps as the proof of \cite[Lemma 3.5]{hutzenthaler2012strong}, by using Lemma \ref{lem3} and Lemma \ref{lem4}.
\end{proof}


\subsection{Proof Theorem \ref{theo2}}
\begin{proof}

	From Lemma \ref{lem7}, we can see that  $X_n=u_n+\hat{v}^N(u_n), ~n\in \{0,\cdot\cdot\cdot, N\}.$\\
 We therefore have
	\begin{align}\label{equa32}
		\|X_n^N\|^p&=\|u_n+\hat{v}^N(t_n,u_n)\|^p\nonumber\\
		       &=\|u_n+\hat{v}^N(t_n,u_n)-\hat{v}^N(t_n,0)+\hat{v}^N(t_n,0)\|^p\nonumber\\
		       &\leq 3^p\|u_n\|^p+3^p\|\hat{v}^N(t_n,u_n)-\hat{v}^N(t_n,0)\|^p+3^p\|\hat{v}^N(t_n,0)\|^p\nonumber\\
		        &\leq 3^p(1+L_{\hat{v}})\|u_n\|^p+3^p\|\hat{v}^N(t_n,0)\|^p.
	\end{align}
	
		Since the function $\hat{v}^N$ is continuous with respect to the variable $t$, there exists a positive constant C such that 
	$$\sup_{0\leq t\leq T} \|\hat{v}^N(t,0)\|\leq C.~~~~~~~~~~~~~~~~~~~~~~~~~~~~~~~~~~~~~~~~~~~~~~~~~~~~~~~~~~$$
	By taking the expectation and the supremum on $n$, the inequality \eqref{equa32} becomes
		\begin{align}\label{equa33a}
	&	\sup_{n\in \{ 0,1,2,...,N\}} \mathbb{E}\|X_n^N\|^p\leq  3^p(1+L_{\hat{v}}) \sup_{n\in \{ 0,1,2,...,N\}} \mathbb{E}\|u_n\|^p+C.~~~~~~~~~~~~~~~~~~~~~~~~~~~~~~~~~~~~
\end{align}

	Let us estimate $\sup_{n\in \{ 0,1,2,...,N\}}\mathbb{E}\|u_n\|^p.$\\
 
	From Lemma \ref{lem7}, for all $n=0,1,...,N-1,~N\in\mathbb{N},$ we have
		\begin{equation}\label{equa33}
		\left \{
		\begin{array}{c c c}
			u_{n+1}-u_n =& P'(t_n)\left[ u_{n+1}+\hat{v}^N(t_{n+1},u_{n+1})\right] h+A^-(t_n) B(t_n)\left[u_{n+1}+\hat{v}^N(t_{n+1},u_{n+1})\right]h\\
			\\
			&+\cfrac{A^-(t_n)f(t_n,u_{n}+\hat{v}^N(t_{n},u_{n}))h}{1+h\|f(t_n,u_{n}+\hat{v}^N(t_{n},u_{n}))\|} +A^-(t_n)g(t_n,u_{n}+\hat{v}^N(t_{n},u_{n}))\Delta W_n,\\
			\\
			\hat{v}^N(t_{n+1},u_{n+1})=&-(A(t_n)+R(t_n)B(t_n))^{-1}\left(R(t_n)B(t_n)u_{n+1}+\cfrac{R(t_n)f(t_n,u_{n}+\hat{v}^N(t_{n},u_{n}))}{1+h\|f(t_n,u_{n}+\hat{v}^N(t_{n},u_{n}))\|} \right).
		\end{array}\right. 
	\end{equation}
	  \newpage
	By substituting the second equation into the first equation in \eqref{equa33} (replacing $\hat{v}^N(t_{n+1,u_n+1})$), we obtain
	\begin{align}\label{equa35a}
			u_{n+1} &=u_n+ P'(t_n) u_{n+1}h-P'(t_n)(A(t_n)+R(t_n)B(t_n))^{-1}\nonumber\\
			&\times\left(R(t_n)B(t_n)u_{n+1}+\cfrac{R(t_n)f(t_n,u_{n}+\hat{v}^N(t_{n},u_{n}))}{1+h\|f(t_n,u_{n}+\hat{v}^N(t_{n},u_{n}))\|} \right)h\nonumber\\
			&+A^-(t_n) B(t_n)u_{n+1}h-A^-(t_n) B(t_n)(A(t_n)+R(t_n)B(t_n))^{-1}\nonumber \\
			&\times\left(R(t_n)B(t_n)u_{n+1}+\cfrac{R(t_n)f(t_n,u_{n}+\hat{v}^N(t_{n},u_{n}))}{1+h\|f(t_n,u_{n}+\hat{v}^N(t_{n},u_{n}))\|} \right)h\nonumber\\
		&+\cfrac{A^-(t_n)f(t_n,u_{n}+\hat{v}^N(t_{n},u_{n}))h}{1+h\|f(t_n,u_{n}+\hat{v}^N(t_{n},u_{n}))\|} +A^-(t_n)g(t_n,u_{n}+\hat{v}^N(t_{n},u_{n}))\Delta W_n.
	\end{align}
  
	Let us set
	\begin{align*}
 M_1(t_n)&=-P'(t_n)+P'(t_n)[A(t_n)+R(t_n)B(t_n)]^{-1}R(t_n)B(t_n)\\
 &-A^-(t_n) B(t_n)+A^-(t_n) B(t_n)[A(t_n)+R(t_n)B(t_n)]^{-1}R(t_n)B(t_n)\\
 \\
 M_2(t_n)&=A^-(t_n)-P'(t_n)(A(t_n)+R(t_n)B(t_n))^{-1}R(t_n)\\
 &-A^-(t_n) B(t_n)(A(t_n)+R(t_n)B(t_n))^{-1}R(t_n).
			\end{align*}
			 
The equation \eqref{equa35a} becomes			
		\begin{align*}
	(I_{d}+hM_1(t_n))	u_{n+1} &=u_n+M_2(t_n)\frac{f(t_n,u_{n}+\hat{v}^N(t_{n},u_{n}))}{1+h\|f(t_n,u_{n}+\hat{v}^N(t_{n},u_{n}))\|} h\nonumber\\
		 &+A^-(t_n)g(t_n,u_{n}+\hat{v}^N(t_{n},u_{n}))\Delta W_n.~~~~~~~~~~~~~~~~~~~
	\end{align*}
	This means that
			\begin{align*}
			u_{n+1} &=(I_{d}+hM_1(t_n))^{-1}u_n+(I_{d}+hM_1(t_n))^{-1}M_2(t_n)\frac{f(t_n,u_{n}+\hat{v}^N(t_{n},u_{n}))}{1+h\|f(t_n,u_{n}+\hat{v}^N(t_{n},u_{n}))\|} h\nonumber\\
		&+(I_{d}+hM_1(t_n))^{-1}A^-(t_n)g(t_n,u_{n}+\hat{v}^N(t_{n},u_{n}))\Delta W_n.
	\end{align*}
 Taking the $\mathbb{R}^d$ norm in both sides and using Assumption \ref{asum2} (A2.3) yields
 			\begin{align*}
			\|u_{n+1} \|&\leq\left|(I_{d}+hM_1(t_n))^{-1}\right|_1\|u_n\|\\
   &+\left|(I_{d}+hM_1(t_n))^{-1}\right|_1\left|M_2(t_n)\right|_F\frac{\|f(t_n,u_{n}+\hat{v}^N(t_{n},u_{n}))\|h}{1+h\|f(t_n,u_{n}+\hat{v}^N(t_{n},u_{n}))\|} \nonumber\\
		&+\left|(I_{d}+hM_1(t_n))^{-1}\right|_1\left|A^-(t_n)\right|_F\|g(t_n,u_{n}+\hat{v}^N(t_{n},u_{n}))\Delta W_n\|\\
  &\leq e^{Kh}\|u_n\|+Ke^{Kh}\frac{\|f(t_n,u_{n}+\hat{v}^N(t_{n},u_{n}))\|h}{1+h\|f(t_n,u_{n}+\hat{v}^N(t_{n},u_{n}))\|} \nonumber\\
		&+e^{Kh}K\|g(t_n,u_{n}+\hat{v}^N(t_{n},u_{n}))\Delta W_n\|.\\
	\end{align*}
	For $n=0,1,2, ...,N-1$, using the fact that $$\exp{(NKh)}=\exp{(NK\frac{T}{N})}=\exp{(KT)}  $$ $$ \text{ and    }\quad \frac{h\left\|f(t_n,u_{n}+\hat{v}^N(t_{i},u_{i}))\right	\|}{1+h\|f(t_n,u_{n}+\hat{v}^N(t_{n},u_{n}))\|} \leq 1,~~ \text{ we obtain }~~~~~~~~~~~~~~~~~~~~~~~~~~~$$ 
				\begin{align*}
		\|u_{n}\| &\leq e^{KT}\|u_0\|+\sum_{i=0}^{n-1}e^{Kh}K+Ke^{Kh}\sum_{i=0}^{n-1}\|g(t_i,u_{i}+\hat{v}^N(t_{i},u_{i}))\Delta W_i\|.~~~~~~~~~~~~~~~~~~
	\end{align*}
	Taking the $L^p(\Omega, \mathbb{R})$ we have
 \begin{align*}
		\left\|	\|u_{n}\|\right\|_{L^p(\Omega, \mathbb{R})} &\leq e^{KT}\left\|\|u_0\|\right\|_{L^p(\Omega, \mathbb{R})} +Ke^{Kh}N+Ke^{Kh}\left\|\sum_{i=0}^{n-1}\|g(t_i,u_{i}+\hat{v}^N(t_{i},u_{i}))\Delta W_i\|\right\|_{L^p(\Omega, \mathbb{R})}.
	\end{align*}
	 	Using Lemma \ref{lem8} yields
    \begin{align*}
		\left\|	u_{n}\right\|_{L^p(\Omega, \mathbb{R}^d)} &\leq e^{KT}\left\|u_0\right\|_{L^p(\Omega, \mathbb{R}^d)} +Ke^{Kh}N+Kpe^{Kh}\left(\sum_{i=0}^{n-1}\|g(t_i,u_{i}+\hat{v}^N(t_{i},u_{i}))\|^2_{L^p(\Omega, \mathbb{R})} \frac{T}{N}\right)^{\frac{1}{2}}.~~~~
	\end{align*}
	 	Using the inequality \eqref{g} and the global Lipschitz on $\hat{v}^N$ yields
      \begin{align*}
		\left\|	u_{n}\right\|_{L^p(\Omega, \mathbb{R}^d)} &\leq e^{KT}\left\|u_0\right\|_{L^p(\Omega, \mathbb{R}^d)} +KNe^{Kh}\\
		&+Kpe^{Kh}\left(C_1\frac{T}{N}\sum_{i=0}^{n-1}(1+4\sup_{0\leq t_i\leq T}\|\hat{v}^N(0,t_i)\|^2+2(2+L_{\hat{v}})\|u_{i}	\|^2_{L^p(\Omega, \mathbb{R}^d)}) \right)^{\frac{1}{2}}\\
  &\leq e^{KT}\left\|u_0\right\|_{L^p(\Omega, \mathbb{R}^d)} +Ke^{Kh}N\\
		&+Kpe^{Kh}C_1^2\left(T(1+4\sup_{0\leq t_i\leq T}\|\hat{v}^N(0,t_i)\|^2)+2(2+L_{\hat{v}})\frac{T}{N}\sum_{i=0}^{n-1}\|u_{i}	\|^2_{L^p(\Omega, \mathbb{R}^d)}) \right)^{\frac{1}{2}}.
	\end{align*}
	 	Let us square both sides
      \begin{align*}
		\left\|	u_{n}\right\|^2_{L^p(\Omega, \mathbb{R}^d)} &\leq  4\left(e^{KT}\left\|u_0\right\|_{L^p(\Omega, \mathbb{R}^d)} +Ke^{Kh}N\right)^2+2KpC_1^2e^{Kh}T(1+4\sup_{0\leq t_i\leq T}\|\hat{v}^N(0,t_i)\|^2)\\
		&+4(2+L_{\hat{v}})KpC_1e^{Kh}\frac{T}{N}\sum_{i=0}^{n-1}\|u_{i}	\|^2_{L^p(\Omega, \mathbb{R}^d)} .~~~~~~~~~~~~~~~~~~~~~~~~~~~~~~~~~~~~~~~~~~~~
	\end{align*}

Using  the discrete Gronwall inequality in Lemma \ref{lem1t1} yields
\begin{align*}
		\left\|	u_{n}\right\|^2_{L^p(\Omega, \mathbb{R}^d)} &\leq  \left[4\left(e^{KT}\left\|u_0\right\|_{L^p(\Omega, \mathbb{R}^d)} +Ke^{Kh}N\right)^2+2Kpe^{Kh}C_1^2T(1+4\sup_{0\leq t_i\leq T}\|\hat{v}^N(0,t_i)\|^2)\right]~~~~~~~~~~~~\\
  &\times \exp{\left(\sum_{i=0}^{n-1}4e^{Kh}KpC_1(2+L_{\hat{v}})\frac{T}{N}\right)}.
	\end{align*}
Taking the square root on both sides
\begin{align*}
		\left\|	u_{n}\right\|_{L^p(\Omega, \mathbb{R}^d)} &\leq  \left[2e^{KT}\left\|u_0\right\|_{L^p(\Omega, \mathbb{R}^d)} +2Ke^{Kh}N+KpC_1e^{Kh}T\sqrt{(1+4C^2)}\right]\\
        &\times\exp{\left(2KpC_1e^{Kh}(2+L_{\hat{v}})T\right)}.~~~~~
	\end{align*}

From \eqref{equa33a},  we have
\begin{equation*}
 \sup_{n\in \{ 0,1,2,...,N\}}\left\|	X^N_{n}	\right\|_{L^p(\Omega, \mathbb{R}^d)}
\leq 3(1+L_{\hat{v}})(C_2+KNe^{Kh})\exp{\left(Kp2C_1e^{Kh}(2+L_{\hat{v}})T\right)}.~~~~~~~~~~~~~~~~~~~~~~~~~~~~~~~~
\end{equation*}

The rest of the proof uses the same steps as the proof of \cite[Lemma 3.9]{hutzenthaler2012strong}, where Lemma 3.1 is replaced by Lemma \ref{lem2}.
 \end{proof}
Before the next result, let us note that the continuous time interpolation of the discrete numerical scheme in  \eqref{equa3} is given by
\begin{align}\label{equa31a}
	\bar{X}_t^N&=S_{h}A\bar{X}^N(t_{n})+\frac{S_{h}f(\bar{X}^N(t_{n}))( t-t_n)}{1+h\left\|f(\bar{X}^N(t_{n})) \right\| }+S_{h}g(\bar{X}^N(t_{n}))( W(t)-W(t_n)),\\\nonumber
	&\bar{X}_0^N=X_0, ~t\in [t_n;t_{n+1}) , ~n=0,1,2,...,N-1.
\end{align}
Here $S_{h}=(A(t_n)-h B(t_n))^{-1}$ and $\bar{X}^N(t_n)=X_n^N,~ n\in \{0,~1,~...,~N-1\}$.
\section{Strong convergence result}
The main result of this paper is given in the following theorem.
\begin{theo}\label{theo4}
	Let $X_t$ be the exact solution of \eqref{equa1} and 	$\bar{X}_t^N$
	 the discrete-continuous form of the numerical approximation given by \eqref{equa31a}. Under Assumption \ref{asum1} and \ref{asum2},   for all $p \in [1, \infty)$, there exists a constant
	$C >0$ such that
	\begin{equation*}
		\left(\mathbb{E}\sup_{t\in [0,T]}\|X_t-\bar{X}_t^N\|^p\right)^{\frac{1}{p}}\leq Ch^{\frac{1}{2}},~~~h=\frac{T}{N}.~~~~~~~~~~~~~~~~~~~~~~~~~~
	\end{equation*}
\end{theo}
To prove this theorem, we first establish some preliminary results.
\subsection{Preliminary results for strong convergence analysis}
The first lemma provides  several important  bounds of the  numerical solution.

\begin{lem}\label{theo3}
	Let $X_n^N: \Omega \longrightarrow \mathbb{R}^d,~ n\in \{0,1,..,N\}\text{ and } N\in \mathbb{N}$ be  the numerical solution  given by\eqref{equa3}. Suppose that the Assumption 1 holds and let $p \geq 1$, then we have
	\begin{align}
		& \sup_{N\in \mathbb{N}}\sup_{n\in \{ 0,1,2,...,N\}}\mathbb{E}[\|f(X_n^N)\|^p]<+\infty\label{equa29}\\
		& \sup_{N\in \mathbb{N}}\sup_{n\in \{ 0,1,2,...,N\}}\mathbb{E}[|g(X_n^N)|_1^p]<+\infty.\label{equa30}~~~~~~~~~~~~~~~~~~~~~~~~~~~~~~~~~~~~
	\end{align}	
\end{lem}
\begin{proof}
	From Assumption \ref{asum1} (A.3),  for all $X\in \mathbb{R}^d$,  we have
		\begin{align}\label{equa25a}
	\| f(X)\|=	\| f(X)-f(0)+f(0)\|&=\| f(X)-f(0)\|+\|f(0)\|\nonumber\\
		&  \leq C\left(1+\| X\|^c \right) \| X\|+\|f(0)\|\nonumber\\
		& \leq C\| X\|  +C\| X\|^{c+1} +\|f(0)\|.
	\end{align}
	We distinguish 2 cases
	\begin{itemize}
		\item If $\| X\|\geq 1$ then $\| X\|\leq \| X\|^{c+1}$, and the equation \eqref{equa25a} becomes:
		\begin{align}\label{equa26a}
			\| f(X)\|\leq
			&   2C\| X\|^{c+1} +\|f(0)\|\nonumber\\
				\leq& (\|f(0)\|+2C)(1+  \| X\|^{c+1}).~~~~~~~~~~~~~~~~
		\end{align}
			\item If $\| X\|\leq 1$ then $C\| X\|\leq C $ and  \eqref{equa25a} becomes
		\begin{align}\label{equa27a}
			\| f(X)\|\leq
			&  C + C\| X\|^{c+1} +\|f(0)\|\nonumber\\
			\leq&(C+\|f(0)\|)(1+  \| X\|^{c+1})\nonumber\\
				\leq&(2C+\|f(0)\|)(1+  \| X\|^{c+1}).~~~~~~~~~~~~~~~~
		\end{align}
	\end{itemize}
	 By combining \eqref{equa26a} and \eqref{equa27a},  \eqref{equa25a} becomes
		\begin{equation*}
		\| f(X)\|	\leq(2C+\|f(0)\|)(1+  \| X\|^{c+1}).~~~~~~~~
			\end{equation*}
			Using Theorem \ref{theo2} yields
			\begin{align*}
			 \sup_{N\in \mathbb{N}}\sup_{n\in \{ 0,1,2,...,N\}}	\| f(X^N_n)\|_{L^p(\Omega,\mathbb{R}^d)}&\leq(2C+\|f(0)\|)\\
			 &\times\big[1+   \sup_{N\in \mathbb{N}}\sup_{n\in \{ 0,1,2,...,N\}}\| X_n^N\|^{c+1}_{L^{p(c+1)}(\Omega,\mathbb{R}^d)}\big]\\
			 &<\infty.
			\end{align*}
		for all $p\in [1,\infty)$	.
		
		On the other hand, from inequality \eqref{g} we have
		 \begin{align*}
			\left| g(X)\right|_1 &\leq C_1(1+ \left\| X\right\|), 	~\forall   X\in \mathbb{R}^d.~~~~~~~~~~~~~~
		\end{align*}
		 Theorem \ref{theo2} allows to have
			\begin{align*}
			\sup_{N\in \mathbb{N}}\sup_{n\in \{ 0,1,2,...,N\}}	\| g(X^N_n)\|_{L^p(\Omega,\mathbb{R}^d)}&\leq C_1(1+ \sup_{N\in \mathbb{N}}\sup_{n\in \{ 0,1,2,...,N\}}\| X_n^N\|^{c+1}_{L^{p(c+1)}(\Omega,\mathbb{R}^d)})\\
			&<\infty.
		\end{align*}
		for all $p\in [1,\infty)$. This completes the proof of Lemma \ref{theo3}.
\end{proof}

The following lemma is fundamental for our analysis.
	\begin{lem}\label{lem9}
		Let $X_t$ be the unique solution of the equation \eqref{equa1} with $X(t)=u(t)+\hat{v}(t,u(t)), ~t\in [0,T]$ where  the function $\hat{v}$ is global Lipschtiz  with respect to the variable $u$. The function $u $  satisfies the following equation
			\begin{equation}\label{equa311}
			\left \{
			\begin{array}{c c c}
				u(t)-u_0 &= \int_{0}^{t}P'(s)\left[ u(s)+\hat{v}(s,u(s))\right] +A^-(s) B(s)\left[u(s)+\hat{v}(s,u(s))\right]ds\\
				\\
				&+\int_{0}^{t}A^-(s)f(s,u(s)+\hat{v}(s,u(s))ds~~~~~~~~~~~~~~~~~~~~~~~~~~~~~~~~~~~ \\
				\\
				&+ \int_{0}^{t}A^-(s)g(s,u(s)+\hat{v}(s,u(s)))dW(s),~~~~~~~~~~~~~~~~~~~~~~~~~~~~\\
				\\
				\hat{v}(t,u(t))&=-(A(t)+R(t)B(t))^{-1}\left[R(t)B(t)u(t)+f_1(t,u(t))\right],~~\\
				\\
				u_0&=P(0)X_0,~~~~~~~~~~~~~~~~~~~~~~~~~~~~~~~~~~~~~~~~~~~~~~~~~~~~~~~~~~~~~\\
                \hat{v}(0,u_0)&=Q(0)X_0,~~~~~~~~~~~~~~~~~~~~~~~~~~~~~~~~~~~~~~~~~~~~~~~~~~~~~~~~~~~~~
			\end{array}\right. 
		\end{equation}
		where the  matrix $A^{-}$ is the pseudo-inverse matrix of the matrix $A$ and,   $P$ , and  $Q$, and $R$ are projectors matrix associated to $A$, such that $R(t)A(t)=0_{d\times d}, \text{ and } A^-(t)A(t)=P(t)=I-Q(t),~t\in [0,T]$ and 
        \begin{equation}\label{f_1}
           f_1(t,u(t))=R(t)f(t,u(t)+\hat{v}(t,u(t))).
        \end{equation}
	\end{lem}
		\begin{proof}
	We refer the readers to the proof of existence and uniqueness of the solution the SDAEs in  \cite[Theorem  1]{serea2025existence}.
	\end{proof}
 
The following lemma can be found in \cite[Lemma 3.7]{hutzenthaler2012strong} or in \cite[Lemma 3.12]{tambue2019strong}.
\begin{lem}\label{lem11}
	Let $k,m\in \mathbb{N}$ and $Z:[0,T]\times \Omega\longrightarrow \mathbb{R}^{k\times m}	$ be a predictable  stochastic process satisfying $\mathbb{P}\left[\int_{0}^{T}\|Z_s\|^2ds<\infty\right]=1$. Then there exists a constant $C_p$ such that for all $t\in [0,T]$ and $p\in [1,+\infty)$
	$$ \left \|\sup_{0\leq t\leq T}\left\|\int_{0}^{t}Z_sdW_s\right\|\right\|_{L^p(\Omega,\mathbb{R})}\leq C_p\left(\int_{0}^{t}\|Z_s\|^2_{\L^p(\Omega,\mathbb{R}^{k\times m})}ds\right)^{\frac{1}{2}}.~~~~~~~~~~~~~$$
\end{lem}

For all $t\in [0,T]$, we denote  by$\lfloor t\rfloor$ the greatest grid point less then $t$. The continuous form of our scheme \eqref{equa31} 
	\begin{equation}\label{equa31b}
    \left \{
			\begin{array}{c c}
	\bar{u}_t^N =&\bar{u}_{\lfloor t\rfloor}^N+ \int_{{\lfloor t\rfloor}}^{t}P'_{\lfloor s\rfloor}\left[ \bar{u}_{ s}^N+\hat{v}_s^N(\bar{u}^N_{ s})\right] +A^-_{\lfloor s\rfloor} B_{\lfloor s\rfloor}\left[\bar{u}^N_{ s}+\hat{v}^N_s(\bar{u}^N_{ s})\right]ds~~~~\\
	&~~~+\int_{\lfloor t\rfloor}^{t}\cfrac{A^-_{\lfloor s\rfloor}f({\lfloor s\rfloor},u^N_{\lfloor s\rfloor}+\hat{v}^N({\lfloor s\rfloor},u^N_{\lfloor s\rfloor}))}{1+h\|f( {\lfloor s\rfloor},u^N_{\lfloor s\rfloor}+\hat{v}^N({\lfloor s\rfloor},u^N_{\lfloor s\rfloor}))\|}ds
	\\
	& ~~~~+\int_{\lfloor t\rfloor}^{t}A^-(\lfloor s\rfloor)g(\lfloor s\rfloor,u^N_{\lfloor s\rfloor}+\hat{v}^N({\lfloor s\rfloor},u^N_{\lfloor s\rfloor}))dW(s),~
	t\in [0;T].~~~~~~ \\
    \hat{v}^N(t,\bar{u}^N_t)=&-(A_{\lfloor t\rfloor}+R_{\lfloor t\rfloor}B_{\lfloor t\rfloor})^{-1}\left(R_{\lfloor t\rfloor}B_{\lfloor t\rfloor}\bar{u}^N_t+\cfrac{f_1({\lfloor t\rfloor},\bar{u}^N_{\lfloor t\rfloor})}{1+h\|f(\lfloor t\rfloor,\bar{u}^N_{\lfloor t\rfloor}+\hat{v}^N(\lfloor t\rfloor,\bar{u}^N_{\lfloor t\rfloor}))\|} \right)
    \end{array}\right.
\end{equation} 
where $f_1$ is defined from \eqref{f_1}
\begin{lem}\label{lem10}
	Let $\bar{u}^N_t$ be the time continuous approximation given by \eqref{equa31b} then we have the following results
	\begin{align}
			&\sup_{t\in [0,T]}\left(\mathbb{E}\left[\|\bar{u}^N_t-\bar{u}^N_{\lfloor t\rfloor}\|^p\right]\right)^{\frac{1}{p}}\leq C_ph^{\frac{1}{2}},\quad
		\sup_{N\in \mathbb{N}}\sup_{t\in [0,T]}	\left(\mathbb{E}\left[\|\bar{u}^N_t\|^p\right]\right)^{\frac{1}{p}}<\infty,\label{e23}\\&\sup_{t\in [0,T]}\left(\mathbb{E}\left[\|f(\bar{u}^N_t+\hat{v}^N(t,\bar{u}^N_t)-f(\bar{u}^N_{\lfloor t\rfloor}+\hat{v}^N(t,\bar{u}^N_{\lfloor t\rfloor}))\|^p\right]\right)^{\frac{1}{p}}\leq C_ph^{\frac{1}{2}}.		
	\end{align}
\end{lem}
\begin{proof}
	From \eqref{equa31b} we have
	\begin{align*}
		\sup_{t\in [0,T]}\|	\bar{u}^N_t-\bar{u}^N_{\lfloor t\rfloor}\|_{L^p(\Omega,\mathbb{R}^d)} &\leq \sup_{t\in [0,T]}\left\|\int_{{\lfloor t\rfloor}}^{t}P'_{\lfloor s\rfloor}\left[ \bar{u}^N_{ s}+\hat{v}^N_s(\bar{u}^N_{ s})\right] +A^-_{\lfloor s\rfloor} B_{\lfloor s\rfloor}\left[\bar{u}^N_{ s}+\hat{v}^N_s(\bar{u}^N_{ s})\right]ds\right\|_{L^p(\Omega,\mathbb{R}^d)}\nonumber\\
		&+\sup_{t\in [0,T]}\left\|\int_{{\lfloor t\rfloor}}^{t}\frac{A^-_{\lfloor t\rfloor}f({\lfloor s\rfloor},u^N_{\lfloor s\rfloor}+\hat{v}^N({\lfloor s\rfloor},u^N_{\lfloor s\rfloor}))}{1+h\|f( {\lfloor s\rfloor},u^N_{\lfloor s\rfloor}+\hat{v}^N({\lfloor s\rfloor},u^N_{\lfloor s\rfloor}))\|}ds\right\|_{L^p(\Omega,\mathbb{R}^d)}
		\nonumber\\
		& +\sup_{t\in [0,T]}\left\|\int_{{\lfloor t\rfloor}}^{t}A^-{\lfloor s\rfloor}g({\lfloor s\rfloor},u^N_{\lfloor s\rfloor}+\hat{v}^N({\lfloor s\rfloor},u^N_{\lfloor s\rfloor}))dW(s)\right\|_{L^p(\Omega,\mathbb{R}^d)}.
	\end{align*}
	Using Assumption \ref{asum2} and  the Lemma \ref{lem11} yields
	\begin{align*}
		\sup_{t\in [0,T]}\|	\bar{u}^N_t-\bar{u}^N_{\lfloor t\rfloor}\|_{L^p(\Omega,\mathbb{R}^d)} &\leq hK(1+K)\sup_{s\in [0,T]}\left\| \bar{u}^N_{ s}+\hat{v}^N_s(\bar{u}^N_{ s})\right\|_{L^p(\Omega,\mathbb{R}^d)} \nonumber\\
		&+hK\sup_{s\in [0,T]}\left\|\frac{f({\lfloor s\rfloor},u^N_{\lfloor s\rfloor}+\hat{v}^N({\lfloor s\rfloor},u^N_{\lfloor s\rfloor}))}{1+h\|f( {\lfloor s\rfloor},u^N_{\lfloor s\rfloor}+\hat{v}^N({\lfloor s\rfloor},u^N_{\lfloor s\rfloor}))\|}\right\|_{L^p(\Omega,\mathbb{R}^d)}
		\nonumber\\
		& +C_pK\sup_{t\in [0,T]}\left(\int_{{\lfloor t\rfloor}}^{t}\left\|g({\lfloor s\rfloor},u^N_{\lfloor s\rfloor}+\hat{v}^N({\lfloor s\rfloor},u^N_{\lfloor s\rfloor}))\right\|^2_{L^p(\Omega,\mathbb{R}^{d\times d_1})}ds\right)^{\frac{1}{2}}.~~~~~~~~~~
	\end{align*} 
Using the fact that $\dfrac{1}{1+h\|f( {\lfloor s\rfloor},u^N_{\lfloor s\rfloor}+\hat{v}^N({\lfloor s\rfloor},u^N_{\lfloor s\rfloor}))\|}\leq 1$, we have
	\begin{align*}
		\sup_{0\leq t\leq T}	\|	\bar{u}^N_t-\bar{u}^N_{\lfloor t\rfloor}\|_{L^p(\Omega,\mathbb{R}^d)} &\leq hK(1+K)	\sup_{0\leq s\leq T}\left\| \bar{u}^N_{ s}+\hat{v}^N_s(\bar{u}^N_{ s})\right\|_{L^p(\Omega,\mathbb{R}^d)}\nonumber\\
		&+hK\sup_{N\in \mathbb{N}}\left\|f({\lfloor t\rfloor},X^N({\lfloor t\rfloor}))\right\|_{L^p(\Omega,\mathbb{R}^d)}
		\nonumber\\
		& +C_pK\left(\sup_{N\in \mathbb{N}}\left\|g({\lfloor t\rfloor},X^N({\lfloor t\rfloor}))\right\|_{L^p(\Omega,\mathbb{R}^{d\times d_1})}\int_{{\lfloor t\rfloor}}^{t}	ds\right)^{\frac{1}{2}}.~~~~~~
	\end{align*} 
	Using Lemma \ref{theo3} yields
	\begin{align}\label{equap}
		\sup_{0\leq t\leq T}	\|	\bar{u}^N_t-\bar{u}^N_{\lfloor t\rfloor}\|_{L^p(\Omega,\mathbb{R}^d)} &\leq hK(1+K)	\sup_{0\leq s\leq T}\left\| \bar{u}^N_{ s}+\hat{v}^N_s(\bar{u}^N_{ s})\right\|_{L^p(\Omega,\mathbb{R}^d)} \nonumber\\
		&+\sqrt{h}K\sqrt{T}C_1 +C_pK\sqrt{h}C_1.~~~~~~~~~~~~~~~~~~~~~~~~~~~
	\end{align} 
	The following is an estimate of $\left\| \bar{u}^N_{ s}+\hat{v}^N_s(\bar{u}^N_{ s})\right\|_{L^p(\Omega,\mathbb{R}^d)}$
	\begin{align}\label{equap1}
		\left\| \bar{u}^N_{ s}+\hat{v}^N_s(\bar{u}^N_{ s})\right\|_{L^p(\Omega,\mathbb{R}^d)}&=\left\| \bar{u}^N_{ s}+\hat{v}^N_s(\bar{u}^N_{ s})-\hat{v}^N(s,0)+\hat{v}^N(s,0)\right\|_{L^p(\Omega,\mathbb{R}^d)}\nonumber\\
		&\leq \left\| \bar{u}^N_{ s}\right\|_{L^p(\Omega,\mathbb{R}^d)}+\left\|\hat{v}^N_s(\bar{u}^N_{ s})-\hat{v}^N(s,0)\right\|_{L^p(\Omega,\mathbb{R}^d)}+\left\|\hat{v}^N(s,0)\right\|\nonumber\\
		&\leq (1+L_{\hat{v}})\left\| \bar{u}^N_{ s}-u^N_{\lfloor s\rfloor}+u^N_{\lfloor s\rfloor}\right\|_{L^p(\Omega,\mathbb{R}^d)}+\left\|\hat{v}^N(s,0)\right\|\nonumber\\
		&\leq (1+L_{\hat{v}})\left(\left\| \bar{u}^N_{ s}-u^N_{\lfloor s\rfloor}\right\|_{L^p(\Omega,\mathbb{R}^d)}+\left\|u^N_{\lfloor s\rfloor}\right\|_{L^p(\Omega,\mathbb{R}^d)}\right)+\left\|\hat{v}^N(s,0)\right\|.~~~~~~~~~~~~~~
	\end{align}
	Replacing \eqref{equap1} in \eqref{equap} yields
	\begin{align*}
		\sup_{0\leq t\leq T}	\|	\bar{u}^N_t-\bar{u}^N_{\lfloor t\rfloor}\|_{L^p(\Omega,\mathbb{R}^d)} &\leq hK(1+K)	\left[\sup_{0\leq s\leq T}(1+L_{\hat{v}})\left\| \bar{u}^N_{ s}-u^N_{\lfloor s\rfloor}\right\|_{L^p(\Omega,\mathbb{R}^d)}\right.\nonumber\\
		&\left.+(1+L_{\hat{v}})	\sup_{0\leq s\leq T}(\left\|u^N_{\lfloor s\rfloor}\right\|_{L^p(\Omega,\mathbb{R}^d)}+\left\|\hat{v}^N(s,0)\right\|)\right] \nonumber\\
		&+\sqrt{h}K\sqrt{T}C_1 +C_pK\sqrt{h}C_1\\
		&\leq hC\sup_{0\leq s\leq T}\left\| \bar{u}^N_{ s}-\bar{u}^N_{\lfloor s\rfloor}\right\|_{L^p(\Omega,\mathbb{R}^d)}+\sqrt{h}K\sqrt{T}C_1 +C_pK\sqrt{h}C_1.\nonumber\\
		&+hC	\left(\sup_{0\leq s\leq T}(\left\|u^N(s)\right\|_{L^p(\Omega,\mathbb{R}^d)}+\left\|\hat{v}^N(s,0)\right\|)\right).
	\end{align*} 
	From Theorem \ref{theo2}, we have
	\begin{align*}
		\sup_{0\leq t\leq T}	\|	\bar{u}^N_t-\bar{u}^N_{\lfloor t\rfloor}\|_{L^p(\Omega,\mathbb{R}^d)} &\leq
		C\int_{\lfloor t\rfloor}^{t}\sup_{0\leq s\leq T}\left\| \bar{u}^N_{ s}-\bar{u}^N_{\lfloor s\rfloor}\right\|_{L^p(\Omega,\mathbb{R}^d)}ds+\sqrt{h}C	.~~~~~~~~~~~~
	\end{align*}
	Using Gronwall's inequality, we have
	\begin{align}\label{equap3}
		\sup_{0\leq t\leq T}	\|	\bar{u}^N_t-\bar{u}^N_{\lfloor t\rfloor}\|_{L^p(\Omega,\mathbb{R}^d)} &\leq
		\sqrt{h}C	\exp{(T^2C)}.~~~~~~~~~~~~~~~~~~~~~~~~~~~~~~~~~~~~~~
	\end{align}
	
	On the other hand, we have the following:
\begin{align}\label{equap4}
	\sup_{N\in\mathbb{N}}	\sup_{0\leq t\leq T}	\|	\bar{u}^N_t\|_{L^p(\Omega,\mathbb{R}^d)}&=		\sup_{N\in\mathbb{N}}\sup_{0\leq t\leq T}	\|	\bar{u}^N_t-\bar{u}^N_{\lfloor t\rfloor}+\bar{u}^N_{\lfloor t\rfloor}\|_{L^p(\Omega,\mathbb{R}^d)}\nonumber\\
	&\leq	\sup_{N\in\mathbb{N}}\sup_{0\leq t\leq T}	\|	\bar{u}^N_t-\bar{u}_{\lfloor t\rfloor}\|_{L^p(\Omega,\mathbb{R}^d)}+	\sup_{N\in\mathbb{N}}\sup_{0\leq t\leq T}	\|\bar{u}^N_{\lfloor t\rfloor}\|_{L^p(\Omega,\mathbb{R}^d)}\nonumber\\
	 &\leq
		\sqrt{h}C	\exp{(T^2C)}+\sup_{N\in\mathbb{N}}\sup_{n\in \{0,1...,N\}}	\|\bar{u}^N(t_n)\|_{L^p(\Omega,\mathbb{R}^d)}\\
	&	<\infty.
	\end{align}
	
	Finally, from Assumption \ref{asum1} (A.3) we have
	\begin{align*}
		&\left(\mathbb{E}\left[\|f(\bar{u}^N_t+\hat{v}^N(t,\bar{u}^N_t)-f(\bar{u}^N_{\lfloor t\rfloor}+\hat{v}^N(t,\bar{u}^N_{\lfloor t\rfloor}))\|^p\right]\right)^{\frac{1}{p}}=\|f(\bar{X}^N_t)-f(\bar{X}^N_{\lfloor t\rfloor})\|_{L^p(\Omega,\mathbb{R}^d)} \\
		&\leq C(1+\|\bar{X}^N_t\|^c_{L^{2pc}(\Omega,\mathbb{R}^d)}+\|\bar{X}^N_{\lfloor t\rfloor}\|^c_{L^{2pc}(\Omega,\mathbb{R}^d)})\|\bar{X}^N_t-\bar{X}^N_{\lfloor t\rfloor}\|_{L^{2p}(\Omega,\mathbb{R}^d)}\\
		&\leq  C(1+2\sup_{0\leq t\leq T}\|\bar{X}^N_t\|^c_{L^{2pc}(\Omega,\mathbb{R}^d)})\|\bar{X}^N_t-\bar{X}^N_{\lfloor t\rfloor}\|_{L^{2p}(\Omega,\mathbb{R}^d)}.
	\end{align*}
	Taking the supremum and using \eqref{equap3} and \eqref{equap4} we have
		\begin{align*}
		\sup_{0\leq t\leq T}\left(\mathbb{E}\left[\|f(\bar{u}^N_t+\hat{v}^N(t,\bar{u}^N_t)-f(\bar{u}^N_{\lfloor t\rfloor}+\hat{v}^N(t,\bar{u}^N_{\lfloor t\rfloor}))\|^p\right]\right)^{\frac{1}{p}}\leq Ch^{\frac{1}{2}}. ~~~~~~~~~~~~~~~~~~
	\end{align*}
\end{proof}
The next lemma is the approximation between the implicit function $\hat{v}$ and its numerical approximation $\hat{v}^N$.
\begin{lem}\label{v}
Let $\hat{v}$ be the implicit function from inherent regular SDE \eqref{equa311} and $\hat{v}^N$ its numerical approximation defined in  \eqref{equa31b}, 
Then we have the following inequality
  \begin{align*}
     \left( \mathbb{E}\sup_{t\in [0,T]} \|\hat{v}(t,	u(t))- \hat{v}^N(t,\bar{u}^N(t))\|^p\right )^{\frac{1}{p}}
       &\leq Ch+C \left( \mathbb{E}\sup_{t\in [0,T]} \|u(t)-\bar{u}^N(t)\|^p\right )^{\frac{1}{p}},
    \end{align*} 
    C is the constant independent to $n \in \mathbb{N}$.
    \end{lem}

   \begin{proof}
       From \eqref{equa311} and \eqref{equa31b} and using the following composition rule
  $$a_1b_1-a_2b_2=a_1(b_1-b_2)+(a_1-a_2)b_2,$$ 
 \newpage
  we have
      \begin{align}\label{eq532}
     	&\|\hat{v}(t,u(t))-\hat{v}^N(t,	\bar{u}^N(t))\|^p=\nonumber\\
        &\left\|-(A(t)+R(t)B(t))^{-1}R(t)B(t)u(t)-(A(t)+R(t)B(t))^{-1}f_1(t,u(t))\right.\nonumber\\
        &\left.+(A_{\lfloor t\rfloor}+R_{\lfloor t\rfloor}B_{\lfloor t\rfloor})^{-1}R_{\lfloor t\rfloor}B_{\lfloor t\rfloor}\bar{u}^N_t+\cfrac{(A_{\lfloor t\rfloor}+R_{\lfloor t\rfloor}B_{\lfloor t\rfloor})^{-1}f_1({\lfloor t\rfloor},\bar{u}^N_{\lfloor t\rfloor})}{1+h\|f(\lfloor t\rfloor,\bar{u}^N_{\lfloor t\rfloor}+\bar{v}^N(\lfloor t\rfloor,\bar{u}^N_{\lfloor t\rfloor}))\|}\right\|^p\nonumber\\
        &\leq 2^p\left\|(A(t)+R(t)B(t))^{-1}R(t)B(t)u(t)-(A_{\lfloor t\rfloor}+R_{\lfloor t\rfloor}B_{\lfloor t\rfloor})^{-1}R_{\lfloor t\rfloor}B_{\lfloor t\rfloor}\bar{u}^N(t)\right\|^p\nonumber\\
          &+2^p\left\|(A(t)+R(t)B(t))^{-1}f_1(t,u(t))- (A_{\lfloor t\rfloor}+R_{\lfloor t\rfloor}B_{\lfloor t\rfloor})^{-1}f_1({\lfloor t\rfloor},\bar{u}^N_{\lfloor t\rfloor})\right\|^p\nonumber\\
        &\leq 2^{2p}(I+II+III+IV),
         \end{align}
where
\begin{align*}
    I:=& \left|(A(t)+R(t)B(t))^{-1}\right|^p_1\left\|R(t)B(t)u(t)- R_{\lfloor t\rfloor}B_{\lfloor t\rfloor}\bar{u}^N(t)\right\|^p\\
    II:=&\left|(A(t)+R(t)B(t))^{-1}-(A_{\lfloor t\rfloor}+R_{\lfloor t\rfloor}B_{\lfloor t\rfloor})^{-1}\right|^p_1\left\|R_{\lfloor t\rfloor}B_{\lfloor t\rfloor}\bar{u}^N_(t) \right\|^p\\
    III:=&\left|(A(t)+R(t)B(t))^{-1}\right|^p_1\left\|f_1(t,u(t))-f_1({\lfloor t\rfloor},\bar{u}^N_{\lfloor t\rfloor}) \right\|^p\\
    IV:=&\left|(A(t)+R(t)B(t))^{-1}-(A_{\lfloor t\rfloor}+R_{\lfloor t\rfloor}B_{\lfloor t\rfloor})^{-1}\right|^p_1 \left\|f_1({\lfloor t\rfloor},\bar{u}^N_{\lfloor t\rfloor}) \right\|^p.
\end{align*}
Let us estimate each term. For the first term, we have the following
\begin{align*}
    I:=& \left|(A(t)+R(t)B(t))^{-1}\right|^p_1\left\|R(t)B(t)u(t)- R_{\lfloor t\rfloor}B_{\lfloor t\rfloor}\bar{u}^N(t)\right\|^p\\
    \leq & K^p\left|R(t)B(t)\right|^p_1\left\|u(t)-\bar{u}^N(t)\right\|^p+K^p\left|R(t)B(t)- R_{\lfloor t\rfloor}B_{\lfloor t\rfloor}\right|^p_1\left\|\bar{u}^N(t)\right\|^p\\
    \leq & K^{2p}\left\|u(t)-\bar{u}^N(t)\right\|^p+K^{2p}h^p\left\|\bar{u}^N(t)\right\|^p
\end{align*}
We also have
\begin{align*}
     II:=&\left\|(A(t)+R(t)B(t))^{-1}-(A_{\lfloor t\rfloor}+R_{\lfloor t\rfloor}B_{\lfloor t\rfloor})^{-1}\right\|^p\left\|R_{\lfloor t\rfloor}B_{\lfloor t\rfloor}\bar{u}^N(t) \right\|^p\\
     \leq & K^{2p}h^p\left\|\bar{u}^N(t)\right\|^p
\end{align*}
For the estimation of the third term, we have
\begin{align*}
  III:=&\left|(A(t)+R(t)B(t))^{-1}\right|^p_1\left\|f_1(t,u(t))-f_1({\lfloor t\rfloor},\bar{u}^N_{\lfloor t\rfloor}) \right\|^p\\
  \leq& K^pC(1+\left\|u(t)\right\|^{pc}+\left\|\bar{u}^N(t)\right\|^{pc})(\left\|u(t)-\bar{u}^N(t)\right\|^p+\left\|\bar{u}^N(t)-\bar{u}^N_{\lfloor t\rfloor})\right\|^p)
\end{align*}
Finally for the fourth term, we have
\begin{align*}
   IV:=&\left|(A(t)+R(t)B(t))^{-1}-(A_{\lfloor t\rfloor}+R_{\lfloor t\rfloor}B_{\lfloor t\rfloor})^{-1}\right|^p_1 \left\|f_1({\lfloor t\rfloor},\bar{u}^N_{\lfloor t\rfloor}) \right\|^p \\
   \leq & K^ph^p\left\|f_1({\lfloor t\rfloor},\bar{u}^N_{\lfloor t\rfloor}) \right\|^p 
\end{align*}
By inserting the previous estimates in \eqref{eq532}, we have the following result
  \begin{align*}
     	\|\hat{v}(t,u(t))-\hat{v}^N(t,	\bar{u}^N(t))\|^p&\leq  K^{2p}\left(\left\|u(t)-\bar{u}^N(t)\right\|^p+h^p\left\|\bar{u}^N(t)\right\|^p+h^p\left\|f_1({\lfloor t\rfloor},\bar{u}^N_{\lfloor t\rfloor}) \right\|^p\right)\nonumber\\
        &+K^pC(1+\left\|u(t)\right\|^{pc}+\left\|\bar{u}^N(t)\right\|^{pc})(\left\|u(t)-\bar{u}^N(t)\right\|^p+\left\|\bar{u}^N(t)-\bar{u}^N_{\lfloor t\rfloor})\right\|^p).
        \end{align*}

        Using Lemma \ref{lem10} and  Lemma \ref{theo3} yields
        \begin{align*}
			\left(\mathbb{E}\left[\sup_{t\in [0,T]}\|\hat{v}(t,u(t))-\hat{v}^N(t,	\bar{u}^N(t))\|^p\right]\right)^{\frac{1}{p}}&\leq Ch+C\left(\mathbb{E}\left[\sup_{t\in [0,T]}\|u_t-\bar{u}^N_t\|^p\right]\right)^{\frac{1}{p}}
            \end{align*}
\end{proof}

We have enough material to prove our main result Theorem \ref{theo4}.
\subsection{Proof of  the strong convergence in Theorem \ref{theo4}}
\begin{proof}

	Indeed, we have
   
	\begin{align}\label{equa34}
			\left(\mathbb{E}\left[\sup_{t\in [0,T]}\|X_t-\bar{X}_t^N\|^p\right]\right)^{\frac{1}{p}}&=\left(\mathbb{E}\left[\sup_{t\in [0,T]}\|u_t+\hat{v}_t(u_t)-\bar{u}^N_t-\hat{v}^N_t(\bar{u}^N_t)\|^p\right]\right)^{\frac{1}{p}}\nonumber\\
            \leq & 2\left(\mathbb{E}\left[\sup_{t\in [0,T]}\|u_t-\bar{u}^N_t\|^p+\sup_{t\in [0,T]}\|\hat{v}_t(u_t)-\hat{v}^N_t(\bar{u}^N_t)\|^p\right]\right)^{\frac{1}{p}}\nonumber\\
            \leq &2\left(\mathbb{E}\sup_{t\in [0,T]}\|u_t-\bar{u}_t^N\|^p\right)^{\frac{1}{p}}+2\left(\mathbb{E}\sup_{t\in [0,T]}\|\hat{v}_t(u_t)-\hat{v}^N_t(\bar{u}^N_t)\|^p\right)^{\frac{1}{p}}\nonumber\\
			&\leq Ch+C\left(\mathbb{E}\left[\sup_{t\in [0,T]}\|u_t-\bar{u}^N_t\|^p\right]\right)^{\frac{1}{p}}.
	\end{align}
	Let us prove that $$\left(\mathbb{E}\left[\sup_{t\in [0,T]}\|u_t-\bar{u}^N_t\|^p\right]\right)^{\frac{1}{p}}\leq C_1h^{\frac{1}{2}}. ~~~~~~~~~~~~~~~~~~~~~~~~~~~~~~~~~~~~~~~~~~~~~~$$

	Indeed from \eqref{equa311} and \eqref{equa31b}, we have
		\begin{align*}
		u_t-\bar{u}^N_t &= \int_{0}^{t}P'(s)\left[ u(s)+\hat{v}(s,u(s))\right]-P'_{\lfloor s\rfloor}\left[ \bar{u}^N( s)+\hat{v}^N(\bar{u}^N( s))\right]ds\\
	&	+\int_{0}^{t} A^-(s) B(s)\left[u(s)+\hat{v}(s,u(s))\right]-A^-_{\lfloor s\rfloor} B_{\lfloor s\rfloor}\left[\bar{u}^N( s)+\hat{v}^N(\bar{u}^N( s))\right]ds\\
		&+\int_{0}^{t}A^-(s)f(s,u(s)+\hat{v}(s,u(s))-\frac{A^-_{\lfloor s\rfloor}f({\lfloor s\rfloor},u^N_{\lfloor s\rfloor}+\hat{v}^N({\lfloor s\rfloor},u^N_{\lfloor s\rfloor}))}{1+h\|f( {\lfloor s\rfloor},u^N_{\lfloor s\rfloor}+\hat{v}^N({\lfloor s\rfloor},u^N_{\lfloor s\rfloor}))\|}ds\\
		&+\int_{0}^{t}A^-(s)g(s,u( s)+\hat{v}(s,u(s))-A^-_{\lfloor s\rfloor}g({\lfloor s\rfloor},u^N_{\lfloor s\rfloor}+\hat{v}^N({\lfloor s\rfloor},u^N_{\lfloor s\rfloor}))dW(s).~~~~~~~~
	\end{align*}
    \newpage
	Let us apply the It\^o formula to the function $\|	u_t-\bar{u}^N_t\|^2$ 
	\begin{align}\label{equa45}
	\|	u_t-\bar{u}^N_t\|^2&=2\int_{0}^{t}\langle u_s-\bar{u}^N_s,P'(s)\left[ u(s)+\hat{v}(s,u(s))\right]-P'_{\lfloor s\rfloor}\left[ \bar{u}^N( s)+\hat{v}^N(\bar{u}^N( s))\right] \rangle ds\nonumber\\
	&+2\int_{0}^{t}\langle u_s-\bar{u}^N_s,   A^-(s) B(s)\left[u(s)+\hat{v}(s,u(s))\right]-A^-_{\lfloor s\rfloor} B_{\lfloor s\rfloor}\left[\bar{u}^N( s)+\hat{v}^N(\bar{u}^N( s))\right]\rangle ds\nonumber\\
	&+2\int_{0}^{t}\langle u_s-\bar{u}_s, A^-(s)f(s,u(s)+\hat{v}(u(s)))-\frac{A^-_{\lfloor s\rfloor}f({\lfloor s\rfloor},\bar{u}^N_{\lfloor s\rfloor}+\hat{v}^N(\bar{u}^N_{\lfloor s\rfloor}))}{1+h\|f( \bar{u}^N_{\lfloor s\rfloor}+\hat{v}^N({\lfloor s\rfloor},\bar{u}^N_{\lfloor s\rfloor}))\|} \rangle ds\nonumber\\
		&+2\int_{0}^{t}\langle u_s-\bar{u}_s, A^-(s)g(s,u( s)+\hat{v}(u(s)))-A^-_{\lfloor s\rfloor}g({\lfloor s\rfloor},\bar{u}^N_{\lfloor s\rfloor}+\hat{v}^N(\bar{u}^N_{\lfloor s\rfloor}))\rangle dW(s)\nonumber\\
		&+\int_{0}^{t} |A^-(s)g(s,u( s)+\hat{v}(u(s)))-A^-_{\lfloor s\rfloor}g({\lfloor s\rfloor},\bar{u}^N_{\lfloor s\rfloor}+\hat{v}^N(\bar{u}^N_{\lfloor s\rfloor}))|^2_Fds\nonumber\\
		&=B_1+B_2+B_3+B_4+B_5,
	\end{align}
    where
    \begin{align*}
      B_1:=&2\int_{0}^{t}\langle u_s-\bar{u}^N_s,P'(s)\left[ u(s)+\hat{v}(s,u(s))\right]-P'_{\lfloor s\rfloor}\left[ \bar{u}^N( s)+\hat{v}^N(\bar{u}^N( s))\right]ds \rangle\\
      B_2:=&2\int_{0}^{t}\langle u_s-\bar{u}^N_s,   A^-(s) B(s)\left[u(s)+\hat{v}(s,u(s))\right]-A^-_{\lfloor s\rfloor} B_{\lfloor s\rfloor}\left[\bar{u}^N( s)+\hat{v}^N(\bar{u}^N( s))\right]\rangle ds\\
      B_3:=&2\int_{0}^{t}\langle u_s-\bar{u}_s, A^-(s)f(s,u(s)+\hat{v}(u(s))-\frac{A^-_{\lfloor s\rfloor}f({\lfloor s\rfloor},\bar{u}^N_{\lfloor s\rfloor}+\hat{v}^N(\bar{u}^N_{\lfloor s\rfloor}))}{1+h\|f( \bar{u}^N_{\lfloor s\rfloor}+\hat{v}^N({\lfloor s\rfloor},\bar{u}^N_{\lfloor s\rfloor}))\|} \rangle ds\\
      B_4:=&2\int_{0}^{t}\langle u_s-\bar{u}_s, A^-(s)g(s,u( s)+\hat{v}(u(s)))-A^-_{\lfloor s\rfloor}g({\lfloor s\rfloor},\bar{u}^N_{\lfloor s\rfloor}+\hat{v}^N(\bar{u}^N_{\lfloor s\rfloor}))\rangle dW(s)\\
      B_5:=&\int_{0}^{t} |A^-(s)g(s,u( s)+\hat{v}(u(s)))-A^-_{\lfloor s\rfloor}g({\lfloor s\rfloor},\bar{u}^N_{\lfloor s\rfloor}+\hat{v}^N(\bar{u}^N_{\lfloor s\rfloor}))|^2_Fds
    \end{align*}
Let us estimate $B_1, B_2, B_3, B_4$ and $B_5$. Indeed, using the fact that $ 2\langle a,b\rangle\leq \|a\|^2 +\|b\|^2$, we have
\begin{align}\label{equa46}
	B_1=&2\int_{0}^{t}\langle u_s-\bar{u}^N_s,P'(s)\left[ u(s)+\hat{v}(s,u(s))\right]-P'_{\lfloor s\rfloor}\left[ \bar{u}^N( s)+\hat{v}^N(\bar{u}^N(s))\right] \rangle ds\nonumber\\
	\leq& \int_{0}^{t}\|u_s-\bar{u}^N_s\|^2+\|P'(s)\left[u(s)+\hat{v}(s,u(s))\right]- P'_{\lfloor s\rfloor}\left[\bar{u}^N(s)+\hat{v}^N(\bar{u}^N( s))\right]\|^2ds\nonumber\\
    \leq &\int_{0}^{t}\|u_s-\bar{u}^N_s\|^2+\|P'(s)\left[u(s)+\hat{v}(s,u(s))- \bar{u}^N( s)-\hat{v}^N(\bar{u}^N(s))\right]\nonumber\\
    &+(P'(s)-P'_{\lfloor s\rfloor})[\bar{u}^N(s)+\hat{v}^N(\bar{u}^N(s))]\|^2ds\nonumber\\
    \leq& \int_{0}^{t}\|u_s-\bar{u}^N_s\|^2+2K^2\|u(s)+\hat{v}(s,u(s))- \bar{u}^N( s)-\hat{v}^N(\bar{u}^N(s))\|^2\nonumber\\
    &+2K^2|s-\lfloor s\rfloor|^2\|\bar{u}^N(s)+\hat{v}^N(\bar{u}^N(s))\|^2ds\nonumber\\
	\leq &\int_{0}^{t}\|u_s-\bar{u}^N_s\|^2+4K^2\left[ \|u(s)-\bar{u}^N_s\|^2+\|\hat{v}(s,u(s))-\hat{v}^N(\bar{u}^N(s))\|^2\right]\nonumber~~~~~~~~~~~~~~~~~~~\\
    &+2K^2h^2\|\bar{u}^N_s+\hat{v}^N(\bar{u}^N_s)\|^2ds\nonumber\\
	\leq &(1+4K^2)\int_0^t\|u(s)-\bar{u}^N_s\|^2+\|\hat{v}(s,u(s))-\hat{v}^N(\bar{u}^N(s))\|^2ds+2K^2h^2\int_0^t\|\bar{X}^N_{ s}\|^2ds.
\end{align}
Similarly, we obtain 
\begin{align}\label{equa47}
	B_2=&2\int_{0}^{t}\langle u_s-\bar{u}^N_s,   A^-(s) B(s)\left[u(s)+\hat{v}(s,u(s))\right]-A^-_{\lfloor s\rfloor} B_{\lfloor s\rfloor}\left[\bar{u}^N_s+\hat{v}^N(\bar{u}^N_s)\right]\rangle ds\nonumber\\
		\leq &(1+4K^4)\int_0^t\|u(s)-\bar{u}^N_s\|^2+\|\hat{v}(s,u(s))-\hat{v}^N(\bar{u}^N(s))\|^2ds \nonumber\\
        &+2K^4h^2\int_0^t\|\bar{X}^N(s)\|^2ds.
	\end{align}

	\begin{align*}
		B_5=&\int_{0}^{t} |A^-(s)g(s,u( s)+\hat{v}(u(s)))-A^-_{\lfloor s\rfloor}g({\lfloor s\rfloor},\bar{u}^N_{\lfloor s\rfloor}+\hat{v}^N(\bar{u}^N_{\lfloor s\rfloor}))|_F^2ds~~~~~~~~~~~~~~~~~\nonumber\\
	\leq& 2K^2 \int_{0}^{t}|g(s,u( s)+\hat{v}(u(s)))-g({\lfloor s\rfloor},\bar{u}^N_{\lfloor s\rfloor}+\hat{v}^N(\bar{u}^N_{\lfloor s\rfloor}))|_F^2ds\nonumber\\
    &+2 \int_{0}^{t}|A^-(s)-A^-_{\lfloor s\rfloor}|_F^2|g({\lfloor s\rfloor},\bar{u}^N_{\lfloor s\rfloor}+\hat{v}^N(\bar{u}^N_{\lfloor s\rfloor}))|_F^2ds
    \end{align*}
    Using Assumptions \ref{asum1} and \ref{asum2},  we have
    \begin{align}\label{equa48}
		B_5\leq& 2K^2 L^2\int_{0}^{t}\|u( s)+\hat{v}(u(s))-\bar{u}^N_{\lfloor s\rfloor}-\hat{v}^N(\bar{u}^N_{\lfloor s\rfloor})\|^2ds\nonumber\\
        &+h^2K^2\int_0^t|g({\lfloor s\rfloor},\bar{u}^N_{\lfloor s\rfloor}+\hat{v}^N(\bar{u}^N_{\lfloor s\rfloor}))|_F^2ds\nonumber\\
			\leq &4K^2 L^2\int_{0}^{t}\|u( s)-\bar{u}^N_s+\bar{u}^N_s-\bar{u}^N_{\lfloor s\rfloor})\|^2ds\nonumber\\
            &+ 4K^2 L^2\int_{0}^{t}\|\hat{v}(u(s))-\hat{v}^N(\bar{u}^N(s))+\hat{v}^N(\bar{u}^N(s))-\hat{v}^N(\bar{u}^N_{\lfloor s\rfloor}))\|^2ds\nonumber\\
            &+h^2K^2\int_0^t|g({\lfloor s\rfloor},\bar{u}^N_{\lfloor s\rfloor}+\hat{v}^N(\bar{u}^N_{\lfloor s\rfloor}))|_F^2ds\nonumber\\
		\leq &8K^2 L^2\left[\int_{0}^{t}\|u( s)-\bar{u}^N_s\|^2ds+(1+L_{\hat{v}})\int_{0}^{t}\|\bar{u}^N_s-\bar{u}^N_{\lfloor s\rfloor}\|^2ds\right]\nonumber\\
        &+ 8K^2 L^2\int_0^t\|\hat{v}(s,u(s))-\hat{v}^N(\bar{u}^N(s))\|^2ds\\
        &+h^2K^2\int_0^t|g({\lfloor s\rfloor},\bar{u}^N_{\lfloor s\rfloor}+\hat{v}^N(\bar{u}^N_{\lfloor s\rfloor}))|_F^2ds.
	\end{align}
	Using the fact that $\langle a,b+c\rangle=\langle a,b\rangle+\langle a,c\rangle$ yields
	\begin{align}\label{equa49}
		B_3&=2\int_{0}^{t}\langle u_s-\bar{u}^N_s, A^-(s)f(s,u(s)+\hat{v}(u(s))-\frac{A^-_{\lfloor s\rfloor}f({\lfloor s\rfloor},\bar{u}^N_{\lfloor s\rfloor}+\hat{v}^N(\bar{u}^N_{\lfloor s\rfloor}))}{1+h\|f(\bar{ u}^N_{\lfloor s\rfloor}+\hat{v}^N({\lfloor s\rfloor},\bar{u}^N_{\lfloor s\rfloor}))\|} \rangle ds\nonumber\\
		&=2\int_{0}^{t}\langle u_s-\bar{u}^N_s, A^-(s)f(s,u(s)+\hat{v}(u(s)))-A^-(s)f( s,\bar{u}^N_s+\hat{v}^N(\bar{u}^N_s)) \rangle ds\nonumber\\
		&+2\int_{0}^{t}\langle u_s-\bar{u}^N_s,A^-(s)f(s,\bar{u}^N_s+\hat{v}^N(\bar{u}^N_s))-\frac{A^-_{\lfloor s\rfloor}f({\lfloor s\rfloor},\bar{u}^N_{\lfloor s\rfloor}+\hat{v}^N(\bar{u}_{\lfloor s\rfloor}))}{1+h\|f( \bar{u}^N_{\lfloor s\rfloor}+\hat{v}^N({\lfloor s\rfloor},\bar{u}^N_{\lfloor s\rfloor}))\|} \rangle ds\nonumber\\
		&=B_{31}+B_{32}.
	\end{align}
	Using the one-sided Lipschitz condition in Assumption \ref{asum1} (A.2) and the fact that $u(t)=P(t)X(t)$ and $X(t)=u(t)+\hat{v}(t,u(t))$, we have
	\begin{align}\label{equa49a}
		B_{31}&:=2\int_{0}^{t}\langle u_s-\bar{u}^N_s, A^-(s)f(s,u(s)+\hat{v}(u(s)))-A^-(s)f( s,\bar{u}^N_s+\hat{v}^N(\bar{u}^N_s)) \rangle ds\nonumber\\
        &\leq 2\int_{0}^{t}|\langle P(s)X(s)-P(s)\bar{X}^N_s, A^-(s)f(s,u(s)+\hat{v}(u(s)))-A^-(s)f( s,\bar{u}^N_s+\hat{v}^N(\bar{u}^N_s)) \rangle| ds\nonumber\\
        & \leq 2K^2\int_{0}^{t}|\langle X(s)-\bar{X}^N_s, f(s,X(s))-f( s,\bar{X}^N_s) \rangle| ds\nonumber\\
        &\leq 2K^2C\int_{0}^{t}\|X(s)-\bar{X}^N_s\|^2ds=2K^2C\int_{0}^{t}\|u(s)+\hat{v}(u(s))-\bar{u}^N_s-\hat{v}^N(\bar{u}^N_s)\|^2ds\nonumber\\
        &\leq 4K^2C\int_{0}^{t}\|u(s)-\bar{u}^N_s\|^2+\|\hat{v}(s,u(s))-\hat{v}^N(s,\bar{u}^N(s))\|^2ds
	\end{align}
	For $B_{32}$ using the inequality $2\langle a, b\rangle\leq \|a\|^2+\|b\|^2$ yields
	\begin{align*}
	B_{32}	&:=2\int_{0}^{t}\langle u_s-\bar{u}^N_s,A^-(s)f(s,\bar{u}^N_s+\hat{v}^N(\bar{u}^N_s))-\frac{A^-_{\lfloor s\rfloor}f({\lfloor s\rfloor},\bar{u}^N_{\lfloor s\rfloor}+\hat{v}^N(\bar{u}^N_{\lfloor s\rfloor}))}{1+h\|f( \bar{u}^N_{\lfloor s\rfloor}+\hat{v}^N({\lfloor s\rfloor},\bar{u}^N_{\lfloor s\rfloor}))\|} \rangle ds\nonumber\\
	&\leq 2\int_{0}^{t}\langle u_s-\bar{u}^N_s,A^-(s)f(s,\bar{u}^N_s+\hat{v}^N(\bar{u}^N_s))-A^-_{\lfloor s\rfloor}f({\lfloor s\rfloor},\bar{u}^N_{\lfloor s\rfloor}+\hat{v}^N(\bar{u}^N_{\lfloor s\rfloor}))\rangle ds\\
    &+2\int_{0}^{t}\langle u_s-\bar{u}^N_s,h\frac{A^-(s)f(s,\bar{u}^N_s+\hat{v}^N(\bar{u}^N_s))\|f({\lfloor s\rfloor},\bar{u}^N_{\lfloor s\rfloor}+\hat{v}^N(\bar{u}^N_{\lfloor s\rfloor}))\|}{1+h\|f( \bar{u}^N_{\lfloor s\rfloor}+\hat{v}^N({\lfloor s\rfloor},\bar{u}^N_{\lfloor s\rfloor}))\|} \rangle ds
    \end{align*}
    Using the fact that $\dfrac{1}{1+h\|f( \bar{u}^N_{\lfloor s\rfloor}+\hat{v}^N({\lfloor s\rfloor},\bar{u}^N_{\lfloor s\rfloor}))\|}\leq 1 $, and Assumption \ref{asum1} yields
    \begin{align}\label{equa49b}
		B_{32}	&\leq \int_{0}^{t} \|u_s-\bar{u}^N_s\|^2+\|A^-(s)f(s,\bar{u}^N_s+\hat{v}^N(\bar{u}^N_s))-A^-_{\lfloor s\rfloor}f({\lfloor s\rfloor},\bar{u}^N_{\lfloor s\rfloor}+\hat{v}^N(\bar{u}^N_{\lfloor s\rfloor}))\|^2 ds\nonumber\\
	&+\int_{0}^{t}\| u_s-\bar{u}^N_s\|^2+h^2\|A^-(s)f(s,\bar{u}^N_s+\hat{v}^N(\bar{u}^N_s))\|^2\|f({\lfloor s\rfloor},\bar{u}^N_{\lfloor s\rfloor}+\hat{v}^N(\bar{u}^N_{\lfloor s\rfloor}))\|^2 ds\nonumber\\
		&\leq 2\int_{0}^{t} \|u_s-\bar{u}^N_s\|^2+K^2\|f(s,\bar{u}^N_s+\hat{v}^N(\bar{u}^N_s))-f({\lfloor s\rfloor},\bar{u}^N_{\lfloor s\rfloor}+\hat{v}^N(\bar{u}^N_{\lfloor s\rfloor}))\|^2 ds\nonumber\\
          &+2h^2K^2\int_{0}^{t}\|f({\lfloor s\rfloor},\bar{u}^N_{\lfloor s\rfloor}+\hat{v}^N(\bar{u}^N_{\lfloor s\rfloor}))\|^2ds\nonumber\\
	&+h^2K^2\int_{0}^{t}\|f({\lfloor s\rfloor},\bar{u}^N_{\lfloor s\rfloor}+\hat{v}^N(\bar{u}^N_{\lfloor s\rfloor}))\|^4 +\|f(s,\bar{u}^N_s+\hat{v}^N(\bar{u}^N_s))\|^4ds.
	\end{align}
	Combining \eqref{equa49a} and \eqref{equa49b} in \eqref{equa49} allows to have 
		\begin{align}\label{equa50}
		B_3&\leq 2(1+2K^2C)\int_{0}^{t}\left\|u_s-\bar{u}_s^N\right\|^2 +\|\hat{v}(s,u(s))-\hat{v}^N(\bar{u}^N(s))\|^2  ds\nonumber\\ &+2K^2\int_{0}^{t}\|f(s,\bar{u}^N_s+\hat{v}^N(\bar{u}^N_s))-f({\lfloor s\rfloor},\bar{u}^N_{\lfloor s\rfloor}+\hat{v}^N(\bar{u}^N_{\lfloor s\rfloor}))\|^2ds~~~~~~~~~~~~~~~~~~~~~~~~~~\nonumber\\
          &+2h^2K^2\int_{0}^{t}\|f({\lfloor s\rfloor},\bar{u}^N_{\lfloor s\rfloor}+\hat{v}^N(\bar{u}^N_{\lfloor s\rfloor}))\|^2ds\nonumber\\
		&+h^2K^2\int_{0}^{t}\|f({\lfloor s\rfloor},\bar{u}^N_{\lfloor s\rfloor}+\hat{v}^N(\bar{u}^N_{\lfloor s\rfloor}))\|^4+\|f( s,\bar{u}^N_s+\hat{v}^N(\bar{u}^N_s))\|^4 ds.
		\end{align}
		By inserting \eqref{equa46}, \eqref{equa47}, \eqref{equa48} and \eqref{equa50} in \eqref{equa45},
		 we obtain 
		\begin{align}\label{equa53}
		\|	u_t-\bar{u}^N_t\|^2&\leq 	C_1\int_{0}^{t}\|u_s-\bar{u}^N_s\|^2ds+C_1\int_{0}^{t}\|\bar{u}^N_s-\bar{u}^N_{\lfloor s\rfloor}\|^2ds\nonumber\\
        &+C_1h^2\int_0^t \|\bar{X}^N(s)\|^2+\|g({\lfloor s\rfloor},\bar{u}^N_{\lfloor s\rfloor}+\hat{v}^N(\bar{u}^N_{\lfloor s\rfloor}))\|^2ds\nonumber\\
	 &+2h^2K^2\int_{0}^{t}\|f({\lfloor s\rfloor},\bar{u}^N_{\lfloor s\rfloor}+\hat{v}^N(\bar{u}^N_{\lfloor s\rfloor}))\|^2ds+C_1\int_0^t\|\hat{v}(s,u(s))-\hat{v}^N(\bar{u}^N(s))\|^2ds\nonumber\\
		&+h^2K^2\int_{0}^{t}\|f({\lfloor s\rfloor},\bar{u}^N_{\lfloor s\rfloor}+\hat{v}^N(\bar{u}^N_{\lfloor s\rfloor}))\|^4+\|f( s,\bar{u}^N_s+\hat{v}^N(\bar{u}^N_s))\|^4 ds\nonumber\\
        &+2K^2\int_{0}^{t}\|f(s,\bar{u}^N_s+\hat{v}^N(\bar{u}^N_s))-f({\lfloor s\rfloor},\bar{u}^N_{\lfloor s\rfloor}+\hat{v}^N(\bar{u}^N_{\lfloor s\rfloor}))\|^2ds~~~~~~~~~~~~~~~~~~~~~~~~~~\nonumber\\
		&+2\int_{0}^{t}\langle u_s-\bar{u}^N_s, A^-(s)g(s,u( s)+\hat{v}(u(s)))-A^-_{\lfloor s\rfloor}g({\lfloor s\rfloor},\bar{u}^N_{\lfloor s\rfloor}+\hat{v}^N(\bar{u}^N_{\lfloor s\rfloor}))\rangle dW(s).
	\end{align}
		Taking the supremum, we obtain
				\begin{align}\label{equa54}
			\sup_{0\leq s \leq t}\|	u_s-\bar{u}^N_s\|^2&\leq 	C_1\int_{0}^{t}\|u_s-\bar{u}^N_{s}\|^2ds+ C_1\int_{0}^{t}\|\bar{u}^N_s-\bar{u}^N_{\lfloor s\rfloor}\|^2ds \nonumber\\
            &+C_1h^2\int_{0}^{t}\|\bar{X}^N_s\|^2+\|g(\lfloor s\rfloor,\bar{u}^N_{\lfloor s\rfloor}+\hat{v}^N(\bar{u}^N_{\lfloor s\rfloor}))\|^2ds\nonumber\\
			&+K^2\int_{0}^{t}\|f(s,\bar{u}^N_s+\hat{v}^N(\bar{u}^N_s))-f({\lfloor s\rfloor},\bar{u}^N_{\lfloor s\rfloor}+\hat{v}^N(\bar{u}^N_{\lfloor s\rfloor}))\|^2ds~~~~~~~~~~~~~~~~~~\nonumber\\
			 &+2h^2K^2\int_{0}^{t}\|f({\lfloor s\rfloor},\bar{u}^N_{\lfloor s\rfloor}+\hat{v}^N(\bar{u}^N_{\lfloor s\rfloor}))\|^2ds+C_1\int_0^t\|\hat{v}(s,u(s))-\hat{v}^N(\bar{u}^N(s))\|^2ds\nonumber\\
		&+h^2K^2\int_{0}^{t}\|f({\lfloor s\rfloor},\bar{u}^N_{\lfloor s\rfloor}+\hat{v}^N(\bar{u}^N_{\lfloor s\rfloor}))\|^4+\|f( s,\bar{u}^N_s+\hat{v}^N(\bar{u}^N_s))\|^4 ds+\sup_{0\leq s\leq t}\left|B_4\right|.
		\end{align}
		Taking the $\|\|_{L^{\frac{p}{2}}(\Omega,\mathbb{R})}$ norm in  both sides for all $p\geq 2 $, and using the H\"older inequality yields
				\begin{align*}
		\left\|	\sup_{0\leq s\leq t}\|	u_s-\bar{u}^N_s\|^2\right\|_{L^{\frac{p}{2}}(\Omega,\mathbb{R})}&\leq 	C_1	(p)(\int_{0}^{t}\left\|\|u_s-\bar{u}^N_s\|^2\right\|_{L^{\frac{p}{2}}(\Omega,\mathbb{R})}+ h^2\left\|\|g(\lfloor s\rfloor,\bar{u}^N_{\lfloor s\rfloor}+\hat{v}^N(\bar{u}^N_{\lfloor s\rfloor}))\|^2\right\|_{L^{\frac{p}{2}}(\Omega,\mathbb{R})}ds)\nonumber\\
       & +C_1	(p)\int_{0}^{t}\left\|\|\hat{v}(s,u(s))-\hat{v}^N(\bar{u}^N(s))\|^2\right\|_{L^{\frac{p}{2}}(\Omega,\mathbb{R})}ds\nonumber\\
		&+ C_1(p)\int_{0}^{t}\left\|\|\bar{u}^N_s-\bar{u}^N_{\lfloor s\rfloor}\|^2\right\|_{L^{\frac{p}{2}}(\Omega,\mathbb{R})}+h^2\left\|\|\bar{X}^N_s\|^2\right\|_{L^{\frac{p}{2}}(\Omega,\mathbb{R})}ds\nonumber\\
			&+K^2C_1(p)\int_{0}^{t}\left\|\|f(s,\bar{u}^N_s+\hat{v}^N(\bar{u}^N_s))-f({\lfloor s\rfloor},\bar{u}^N_{\lfloor s\rfloor}+\hat{v}^N(\bar{u}^N_{\lfloor s\rfloor}))\|^2\right\|_{L^{\frac{p}{2}}(\Omega,\mathbb{R})}ds\nonumber\\
			&+h^2C_1(p)\int_{0}^{t}\left\|\|f(\lfloor s\rfloor,\bar{u}^N_{\lfloor s\rfloor}+\hat{v}^N(\bar{u}^N_{\lfloor s\rfloor}))\|^4 \right\|_{L^{\frac{p}{2}}(\Omega,\mathbb{R})}ds\nonumber\\
            &+h^2C_1(p)\int_{0}^{t}\left\|\|f( s,\bar{u}^N_s+\hat{v}^N(\bar{u}^N_s))\|^4 \right\|_{L^{\frac{p}{2}}(\Omega,\mathbb{R})}ds\nonumber\\
			&+h^2C_1(p)\int_{0}^{t}\left\|\|f(\lfloor s\rfloor,\bar{u}^N_{\lfloor s\rfloor}+\hat{v}^N(\bar{u}^N_{\lfloor s\rfloor}))\|^2 \right\|_{L^{\frac{p}{2}}(\Omega,\mathbb{R})}ds+\left\|\sup_{0\leq s\leq t}\left|B_4\right|\right\|_{L^{\frac{p}{2}}(\Omega,\mathbb{R})}.
		\end{align*}
		Since
		\begin{align*}
				\left\|	\sup_{0\leq s\leq t}\|	u_s-\bar{u}^N_s\|^2\right\|_{L^{\frac{p}{2}}(\Omega,\mathbb{R})}=&\left(\mathbb{E}\sup_{0\leq s\leq t}\left(\|	u_s-\bar{u}^N_s\|\right)^{p}\right)^{\frac{2}{p}}\\
				=&\left(\left[\mathbb{E}\sup_{0\leq s\leq t}\left(\|	u_s-\bar{u}^N_s\|\right)^{p}\right]^{\frac{1}{2}}\left[\mathbb{E}\sup_{0\leq s\leq t}\left(\|	u_s-\bar{u}^N_s\|\right)^{p}\right]^{\frac{1}{2}}\right)^{\frac{2}{p}}\\
					=&\left(\mathbb{E}\sup_{0\leq s\leq t}\left(\|	u_s-\bar{u}^N_s\|\right)^{p}\right)^{\frac{1}{p}}\left(\mathbb{E}\sup_{0\leq s\leq t}\left(\|	u_s-\bar{u}^N_s\|\right)^{p}\right)^{\frac{1}{p}}\\
					=&\left\|\sup_{0\leq s\leq t}\|	u_s-\bar{u}^N_s\|\right\|_{L^{p}(\Omega,\mathbb{R})}\left\|\sup_{0\leq s\leq t}\|	u_s-\bar{u}^N_s\|\right\|_{L^{p}(\Omega,\mathbb{R})}
					=\left\|\sup_{0\leq s\leq t}\|	u_s-\bar{u}^N_s\|\right\|^2_{L^{p}(\Omega,\mathbb{R})},
		\end{align*}
		We therefore have
        	\begin{align}\label{equa56a}
		\left\|\sup_{0\leq s\leq t}\|	u_s-\bar{u}^N_s\|\right\|^2_{L^{p}(\Omega,\mathbb{R})}
			&\leq 	C_1	(p)\int_{0}^{t}\left\|\|u_s-\bar{u}^N_s\|\right\|^2_{L^{p}(\Omega,\mathbb{R})}ds+	C_1(p)\int_{0}^{t}\left\|        \|\hat{v}(s,u(s))-\hat{v}^N(\bar{u}^N(s))\|   \right\|^2_{L^{p}(\Omega,\mathbb{R})}ds\nonumber\\
		&+ C_	1(p)\int_{0}^{t}\left\|\|\bar{u}^N_s-\bar{u}^N_{\lfloor s\rfloor}\|\right\|^2_{L^{p}(\Omega,\mathbb{R})}+h^2\left\|\|\bar{X}^N_s\|\right\|^2_{L^{p}(\Omega,\mathbb{R})}ds\nonumber\\
        &+C_1(p)h^2\int_{0}^{t}\left\|\|g(\lfloor s\rfloor,\bar{u}^N_{\lfloor s\rfloor}+\hat{v}^N(\bar{u}^N_{\lfloor s\rfloor}))\|\right\|^2_{L^{p}(\Omega,\mathbb{R})}ds\nonumber\\
			&+C_1(p)\int_{0}^{t}\left\|\|f(s,\bar{u}^N_s+\hat{v}^N(\bar{u}^N_s))-f({\lfloor s\rfloor},\bar{u}^N_{\lfloor s\rfloor}+\hat{v}^N(\bar{u}^N_{\lfloor s\rfloor}))\|\right\|^2_{L^{p}(\Omega,\mathbb{R})}ds\nonumber\\
			&+h^2C_1(p)\int_{0}^{t}\left\|\|f(\lfloor s\rfloor,\bar{u}^N_{\lfloor s\rfloor}+\hat{v}^N(\bar{u}^N_{\lfloor s\rfloor}))\|\right\|^2_{L^{2p}(\Omega,\mathbb{R})}ds\nonumber\\
              &+C_1(p)h^2\int_{0}^{t}\left\|\|f(\lfloor s\rfloor,\bar{u}^N_{\lfloor s\rfloor}+\hat{v}^N(\bar{u}^N_{\lfloor s\rfloor}))\|\right\|^2_{L^{p}(\Omega,\mathbb{R})}ds
            +\left\|\sup_{0\leq s\leq t}\left|B_4\right|\right\|_{L^{\frac{p}{2}}(\Omega,\mathbb{R})}.
		\end{align}
        Using  Lemma \ref{v}, Lemma \ref{lem10} and Lemma \ref{theo3},  the inequality \eqref{equa56a} becomes
        		
			\begin{align}\label{equa56}
		\left\|\sup_{0\leq s\leq t}\|	u_s-\bar{u}^N_s\|\right\|^2_{L^{p}(\Omega,\mathbb{R})}
			&\leq 	C_1	(p)\int_{0}^{t}\left\|\|u_s-\bar{u}^N_s\|\right\|^2_{L^{p}(\Omega,\mathbb{R})}ds+ C_1h
            +\left\|\sup_{0\leq s\leq t}\left|B_4\right|\right\|_{L^{\frac{p}{2}}(\Omega,\mathbb{R})}.
		\end{align}

    Let us  bound $\left\|\sup_{0\leq s\leq t}\left|B_4\right|\right\|_{L^{\frac{p}{2}}(\Omega,\mathbb{R})}$ by using Lemma \ref{lem11},
		\begin{align*}
			\left\|\sup_{0\leq s\leq t}\left|B_4\right|\right\|_{L^{\frac{p}{2}}(\Omega,\mathbb{R})}
	&:=	\left\|\sup_{0\leq s\leq t}\left|2\int_{0}^{t}\langle u_s-\bar{u}^N_s, A^-(s)g(s,u( s)+\hat{v}(u(s)))\right.\right.\nonumber\\
	&\left.\left.-A^-_{\lfloor s\rfloor}g({\lfloor s\rfloor},\bar{u}^N_{\lfloor s\rfloor}+\hat{v}^N(\bar{u}^N_{\lfloor s\rfloor}))\rangle dW(s)\right|\right\|_{L^{\frac{p}{2}}(\Omega,\mathbb{R})}\nonumber\\
	&\leq C_p\left(\int_{0}^{t}\left\|2\langle u_s-\bar{u}^N_s, A^-(s)g(s,u( s)+\hat{v}(u(s)))\right.\right.\nonumber~~~~~~~~~~~~~~~~~~~~~~~~\\
	&\left.\left.-A^-_{\lfloor s\rfloor}g({\lfloor s\rfloor},\bar{u}^N_{\lfloor s\rfloor}+\hat{v}^N(\bar{u}^N_{\lfloor s\rfloor}))\rangle \right\|^2_{L^{\frac{p}{2}}(\Omega,\mathbb{R}^{d})}ds\right)^{\frac{1}{2}}.
		\end{align*}
	Using the fact that  $\|ab\|\leq\|a\|\|b\| $ yields
			\begin{align*}
			\left\|\sup_{0\leq s\leq t}\left|B_4\right|\right\|_{L^{\frac{p}{2}}(\Omega,\mathbb{R})}
			&\leq 2C_p\left(\int_{0}^{t}\left\|\|u_s-\bar{u}^N_s\|\right\|^2_{L^{p}(\Omega,\mathbb{R})}\times\right.\\
		&	\left.\left\|A^-(s)g(s,u( s)+\hat{v}(u(s)))-A^-_{\lfloor s\rfloor}g({\lfloor s\rfloor},\bar{u}^N_{\lfloor s\rfloor}+\hat{v}^N(\bar{u}^N_{\lfloor s\rfloor}))\right\|^2_{L^{\frac{p}{2}}(\Omega,\mathbb{R}^{d\times d_1})}ds
		\right)^{\frac{1}{2}}\\
		&\leq2C_p\left\| \sup_{0\leq s\leq t}\|u_s-\bar{u}^N_s\|\right\|_{L^{p}(\Omega,\mathbb{R})}\times\\
		&\left(\int_{0}^{t} K\left \|g(s,u( s)+\hat{v}(u(s)))-g({\lfloor s\rfloor},\bar{u}^N_{\lfloor s\rfloor}+\hat{v}^N(\bar{u}^N_{\lfloor s\rfloor}))\right\|^2_{L^{p}(\Omega,\mathbb{R}^{d\times d_1})}\right.\\
        &\left.+h^2\left\|g({\lfloor s\rfloor},\bar{u}^N_{\lfloor s\rfloor}+\hat{v}^N(\bar{u}^N_{\lfloor s\rfloor}))\right\|^2_{L^{p}(\Omega,\mathbb{R}^{d\times d_1})}
		ds\right)^{\frac{1}{2}}\\
			&\leq C_1 \left\|\sup_{0\leq s\leq t}\|u_s-\bar{u}^N_s\|\right\|_{L^{p}(\Omega,\mathbb{R})}\left(\int_{0}^{t} \left\|u_ s-\bar{u}^N_{\lfloor s\rfloor}\right\|^2_{L^{p}(\Omega,\mathbb{R}^{d})}\right.\\
             &\left.+\left\|\hat{v}(u_ s)-\hat{v}^N(\bar{u}^N_{\lfloor s\rfloor})\right\|^2_{L^{p}(\Omega,\mathbb{R}^{d})}+h^2\left\|g({\lfloor s\rfloor},\bar{u}^N_{\lfloor s\rfloor}+\hat{v}^N(\bar{u}^N_{\lfloor s\rfloor}))\right\|^2_{L^{p}(\Omega,\mathbb{R}^{d\times d_1})}
		ds\right)^{\frac{1}{2}}.
		\end{align*}
		Using the fact that $2ab\leq a^2+b^2$ yields
		\begin{align}\label{equa57}
			\left\|\sup_{0\leq s\leq t}\left|B_4\right|\right\|_{L^{\frac{p}{2}}(\Omega,\mathbb{R})}
				&\leq \frac{1}{2}\left\|\sup_{0\leq s\leq t}\|u_s-\bar{u}^N_s\|\right\|_{L^{p}(\Omega,\mathbb{R})}^2+2C_1\int_{0}^{t} \left\|u_ s-\bar{u}^N_{\lfloor s\rfloor}\right\|^2_{L^{p}(\Omega,\mathbb{R}^{d})}\nonumber\\
                &+h^2\left\|g({\lfloor s\rfloor},\bar{u}^N_{\lfloor s\rfloor}+\hat{v}^N(\bar{u}^N_{\lfloor s\rfloor}))\right\|^2_{L^{p}(\Omega,\mathbb{R}^{d\times d_1})}+\left\|\hat{v}(u_ s)-\hat{v}^N(\bar{u}^N_{\lfloor s\rfloor})\right\|^2_{L^{p}(\Omega,\mathbb{R}^{d})}ds\nonumber\\
                &\leq \frac{1}{2}\left\|\sup_{0\leq s\leq t}\|u_s-\bar{u}^N_s\|\right\|_{L^{p}(\Omega,\mathbb{R})}^2+2C_1\int_{0}^{t} \left\|u_ s-\bar{u}^N_s\right\|^2_{L^{p}(\Omega,\mathbb{R}^{d})}\nonumber\\
                 &+h^2\left\|g({\lfloor s\rfloor},\bar{u}^N_{\lfloor s\rfloor}+\hat{v}^N(\bar{u}^N_{\lfloor s\rfloor}))\right\|^2_{L^{p}(\Omega,\mathbb{R}^{d\times d_1})}+\left\|\hat{v}(u_ s)-\hat{v}^N(\bar{u}^N_s)\right\|^2_{L^{p}(\Omega,\mathbb{R}^{d})}\nonumber\\
                 &+ (1+L_{\hat{v}})\left\|\bar{u}^N_s-\bar{u}^N_{\lfloor s\rfloor}\right\|^2_{L^{p}(\Omega,\mathbb{R}^{d})}ds\nonumber\\
                  &\leq \frac{1}{2}\left\|\sup_{0\leq s\leq t}\|u_s-\bar{u}^N_s\|\right\|_{L^{p}(\Omega,\mathbb{R})}^2+2C_1\int_{0}^{t} \left\|u_ s-\bar{u}^N_s\right\|^2_{L^{p}(\Omega,\mathbb{R}^{d})}ds+C_1h
			\end{align}
			By replacing \eqref{equa57} in \eqref{equa56}, we obtain
           \begin{align*}
		\left\|\sup_{0\leq s\leq t}\|	u_s-\bar{u}^N_s\|\right\|^2_{L^{p}(\Omega,\mathbb{R})}
			&\leq 	C_1	(p)\int_{0}^{t}\left\|\sup_{0\leq s\leq t}\|u_s-\bar{u}^N_s\|\right\|^2_{L^{p}(\Omega,\mathbb{R})}ds+C_1h
		\end{align*}
			Using Gronwall's lemma, we obtain
						
	 \begin{align*}
		\left\|\sup_{0\leq s\leq t}\|	u_s-\bar{u}^N_s\|\right\|^2_{L^{p}(\Omega,\mathbb{R})}
			&\leq 	C_1(p)h\exp{(C(p)T)}.
            \end{align*}
	
	This means that
	\begin{align}\label{equa59}
		\left(\mathbb{E}\left[\sup_{0\leq t\leq T}\|u_t-\bar{u}^N_t\|^p\right]\right)^{\frac{1}{p}}
		\leq C_1(p)h^{\frac{1}{2}}.~~~~~~~~~~~~~~~~
	\end{align}
	Finally,  we replace \eqref{equa59} in \eqref{equa34} and obtain
	\begin{align*}
		\left(\mathbb{E}\left[\sup_{0\leq t\leq T}\|X_t-\bar{X}_t^N\|^p\right]\right)^{\frac{1}{p}}
		&\leq C_1h^{\frac{1}{2}}.
	\end{align*}
Note that the constant $C_1$ changes from line to line.

\end{proof}
\section{ Numerical simulations}
This example aims to check whether the theoretical strong convergence order of scheme \eqref{equa31a}, as stated in Theorem \ref{theo4} is consistent with  simulations. To achieve that aims, we consider the following SDAE

	\begin{equation}\label{solu1}
		\begin{pmatrix}
			1 & 0 &1  \\
			0 & 0& 0\\
			1&0 &0
		\end{pmatrix}dX=[BX(t)+f(t,X(t))]dt+g(t,X(t))dW(t),~~ X(0)=(10^{-2};0;10^{-2}), ~~t\in\left[0,T \right],
	\end{equation}
	with
    $B=\begin{pmatrix}
			0 & 0 &0  \\
			0 & -1& 0\\
			0&0 &1
		\end{pmatrix} ,$
	$f(x_1,x_2,x_3) =\begin{pmatrix}
		x_1  \\
		x_2-x_2^5\\
		x_3-x_3^3
	\end{pmatrix} ,$ and $$ g(x_1,x_2,x_3)=\begin{pmatrix}
	x_1-x_3 & x_2&x_3  \\
		0 & 0& 0\\
		x_2-x_3&	x_1&x_2-x_1
	\end{pmatrix}.$$
	
Let us check that the assumptions used in Theorem \ref{theo1} are satisfied. This will ensure that equation \eqref{solu1} admits a unique solution. We can easily  check  that the pseudo-inverse matrix $A^-$ is computed as
$$A^-=	\begin{pmatrix}
0 & 0 &1  \\
	0 & 0& 0\\
	1&0&-1
\end{pmatrix}.$$
The projector matrix $R$ is  therefore given by $$R=I_{3\times 3}-AA^-=	\begin{pmatrix}
	0 & 0 &0  \\
	0 & 1& 0\\
	0&0&0
\end{pmatrix}.$$
The $X(t)$ can be written as $X(t)=U(t)+V(t)$ and  the constraint equation (AE) is given by
\begin{align*}
	A(t)V(t)+R(t)&[B(U(t)+V(t))+f(U(t)+V(t))]=	
	&\begin{pmatrix}
		v_1(t)+v_3(t)   \\
		-(v_2(t)+u_2(t))^5\\
		v_1(t)
	\end{pmatrix}
	&=\begin{pmatrix}
	0   \\
		0\\
	0
	\end{pmatrix}.\\
\end{align*}
Therefore 
$$v_1(t)=0~~v_3(t)=0 \text{ and } v_2(t)=-u_2(t) ,~\forall~t\in [0,T].$$
 In addition, the noise sources do not appear into the constraints AE  which is globally solvable. This implies that equation \eqref{solu1} is an index-1 SDAE.

 Let is  prove that the function $f(\cdot, \cdot)$  satisfies the Assumption \ref{asum1}. Indeed  we have 
\begin{align*}
&\left\langle (X-Y)^T,f(t,X)-f(t,Y)\right\rangle=\begin{pmatrix}
	x_1-y_1  \\
	  x_2-y_2\\
   x_3-y_3
	\end{pmatrix}\cdot\begin{pmatrix}
	x_1-y_1  \\
		x_2-x_2^5-y_2+y_2^5\\
		x_3-x_3^3-y_3+y_3^3
	\end{pmatrix}~~~~~~~~~~~~~~~~~~~~~~~~~~~~~~\\
	&=(x_1-y_1)^2+(x_2-y_2)^2-(x_2-y_2)(x_2^5-y_2^5)+( x_3-y_3)^2-( x_3-y_3)(x_3^3-y_3^3)\\
    &=\|X-Y\|^2-(x_2-y_2)(x_2^5-y_2^5)-( x_3-y_3)(x_3^3-y_3^3).\\
    &\leq \|X-Y\| ^2.
\end{align*}

In another hand, we also have
\begin{align*}
 \|f(t,X)-f(t,Y)\| ^2&= \|x_1-y_1; x_2-y_2-x_2^5+y_2^5;x_3-y_3-x_3^3+y_3^5\|\\
 &=  |x_1-y_1|^2+ |x_2-y_2-x_2^5+y_2^5|^2+|x_3-y_3-x_3^3+y_3^3|^2\\
 &\leq  |x_1-y_1|^2+2|x_2-y_2|^2+2|x_3-y_3|^2+ 2|-x_2^5+y_2^5|^2+2|-x_3^3+y_3^3|^2\\
& \leq c\|X-Y\|^2(1+\|X\|^k+\|Y\|^k )^2, ~c\geq 0, ~k\geq 4.
\end{align*}
We can  then conclude that the unique solution $X(\cdot)$  of equation \eqref{solu1}  exists. For the numerical simulation we  have taken $N=2^{18}, ~T=1,~ X_0=(10^{-2}; 0;10^{-2}),~ t_0=0$. The  strong error is evaluated with 300 samples. 
 In addition, we can also prove that the Assumption 2 (A2.3) is satisfied. Indeed
 $I_{3} +hM_1(t_n),~n\in\mathbb{N}$ is non singular matrix as
  \begin{align*}
    I_{3} +hM_1(t_n) = \begin{pmatrix}
1 & 0 &0  \\
	0 & 1& 0\\
	0&0&1
\end{pmatrix}+\begin{pmatrix}
0 & 0 &h  \\
	0 & 0& 0\\
	0&0&-h
\end{pmatrix}=\begin{pmatrix}
1 & 0 &-h  \\
	0 & 1& 0\\
	0&0&1+h
\end{pmatrix}
  \end{align*}
  This means that for all $h$ the matrix $I_{3} +hM_1(t_n),~n\in\mathbb{N}$ is non singular matrix and 
   $(I_{3} +hM_1(t_n))^{-1}=\begin{pmatrix}
1 & 0 &\frac{h}{h+1}  \\
	0 & 1& 0\\
	0&0&\frac{1}{1+h}
\end{pmatrix},~n\in\mathbb{N}$ finally 
$$|(I_{3} +hM_1(t_n))^{-1} |_1=\max_{i}\sum_j|a_{ij}|=\frac{1}{h+1}+\frac{h}{1+h}=1\leq \exp{(kh)},\,\,\,\, \forall h$$

 \begin{figure}[h]
	\centering
	\includegraphics[width=1\linewidth]{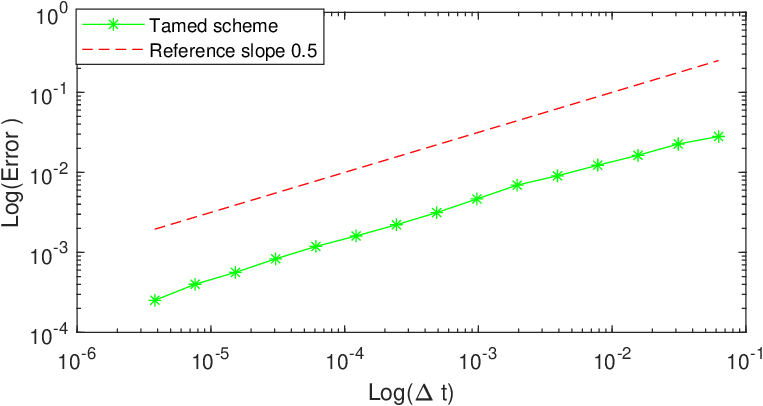}
	\caption{Strong convergence with our novel  semi-implicit tamed  scheme for SDAEs}
	\label{fig:eroorr}
\end{figure}
Figure \ref{fig:eroorr} shows the strong   convergence  in log scale with rate $\frac{1}{2}$. This confirms the theoretical result stated in Theorem \ref{theo4}.


\textbf{Funding:} Open access funding provided by Western Norway University Of Applied
Sciences Not applicable


\end{document}